\documentclass[11pt,reqno]{amsart}
\pdfoutput=1
\usepackage{amsmath,amssymb,amscd,mathrsfs,epic,latexsym,tikz,mathrsfs,cite,stmaryrd}
\usepackage{pb-diagram}
\usepackage[matrix,arrow]{xy}
\usepackage[margin=2cm]{geometry}
\usepackage[skip=3pt,font=footnotesize,labelfont=bf]{caption}
\usepackage{longtable,multirow,setspace}
\usepackage{bm}
\usepackage{mathtools}

\usepackage[enableskew]{youngtab}
\usepackage{genyoungtabtikz}
\usepackage{mfabacus}
\usepackage{color}
\usepackage{tikz-cd}
\usepackage{nicematrix}

\usepackage{comment}
\usepackage[new]{old-arrows}

\definecolor{mediumblue}{rgb}{0.0, 0.0, 0.8}
\colorlet{darkgreen}{green!50!black}
\usepackage{doi}
\usepackage{hyperref}
\usepackage[capitalise]{cleveref}

\hypersetup{
  colorlinks=true,
  linktoc=page,
  citecolor=darkgreen,
  linkcolor=blue,
  urlcolor=mediumblue,
  unicode
}

\usepackage{easybmat}
\usepackage{lscape,marginnote}
\usepackage{ifthen}
\usepackage{xkeyval}
\usepackage{xcolor}
\usepackage{calc}
\usepackage[colorinlistoftodos]{todonotes}

\crefname{theorem}{Theorem}{Theorems}
\crefname{prop}{Proposition}{Propositions}
\crefname{lem}{Lemma}{Lemmas}
\crefname{cor}{Corollary}{Corollaries}
\crefname{conj}{Conjecture}{Conjectures}
\crefname{defn}{Definition}{Definitions}
\crefname{exe}{Example}{Examples}
\crefname{rmk}{Remark}{Remarks}

\Crefname{theorem}{Theorem}{Theorems}
\Crefname{prop}{Proposition}{Propositions}
\Crefname{lem}{Lemma}{Lemmas}
\Crefname{cor}{Corollary}{Corollaries}
\Crefname{conj}{Conjecture}{Conjectures}
\Crefname{defn}{Definition}{Definitions}
\Crefname{exe}{Example}{Examples}
\Crefname{rmk}{Remark}{Remarks}

\definecolor{answercolor}{RGB}{0, 112, 48}
\excludecomment{answer}      

\makeatletter

\numberwithin{equation}{section}

\newtheorem{thm}{Theorem}[section]
\newtheorem{prop}[thm]{Proposition}
\newtheorem{lem}[thm]{Lemma}
\newtheorem{cor}[thm]{Corollary}

\theoremstyle{definition}
\newtheorem{defn}[thm]{Definition}
\newtheorem{exe}[thm]{Example}
\newtheorem{rmk}[thm]{Remark}

\let\<\langle
\let\>\rangle

\newcommand\Comment[2][\relax]{\space\par\medskip\noindent%
   \fbox{\begin{minipage}{\textwidth}\textbf{Comment\ifx\relax#1\else---#1\fi}\newline%
        #2\end{minipage}}\medskip
}

\def\b1{\text{\boldmath$1$}}

\def\bla{\text{\boldmath$\lambda$}}

\newcommand{\la}{\lambda}
\newcommand{\lak}{\la^{+k}}
\newcommand\zez{\mathbb{Z}/e\mathbb{Z}}
\def\aks{\mathcal{H}_{r,n}^{\bm{s}}}
\newcommand{\rem}[2]{\operatorname{rem}_{#1}(#2)}
\newcommand{\add}[2]{\operatorname{add}_{#1}(#2)}
\newcommand{\remab}[2]{\operatorname{rem}_{#1}(#2)\uparrow_{\mathfrak{n}}}
\newcommand{\addab}[2]{\operatorname{add}_{#1}(#2)\uparrow_{\mathfrak{n}}}
\newcommand{\remabp}[2]{\operatorname{rem}_{#1}(#2)\uparrow_{\mathfrak{n}'}}
\newcommand{\addabp}[2]{\operatorname{add}_{#1}(#2)\uparrow_{\mathfrak{n}'}}
\newcommand{\remabs}[2]{\operatorname{rem}_{#1}(#2)\uparrow_{\mathfrak{n}''}}
\newcommand{\addabs}[2]{\operatorname{add}_{#1}(#2)\uparrow_{\mathfrak{n}''}}

\newcommand{\Z}{\mathbb{Z}}

\def\phi{{\varphi}}

\newcommand\UU{\mathcal{U}}

\renewcommand\geq\geqslant
\renewcommand\leq\leqslant

\renewcommand\succeq\succcurlyeq
\renewcommand\preceq\preccurlyeq

\renewcommand{\trianglerighteq}{\trianglerighteqslant}
\renewcommand{\trianglelefteq}{\trianglelefteqslant}

\newcommand{\HC}{{\mathcal H}}

\renewcommand{\mod}{\bmod \,}

\def\b{\mathfrak{b}}
\def\k{\Bbbk}
\def\x{x}

\def\y{y}

{\catcode`\|=\active
  \gdef\set#1{\mathinner{\lbrace\,{\mathcode`\|"8000%
  \let|\midvert #1}\,\rbrace}}
}
\def\midvert{\egroup\mid\bgroup}

\colorlet{darkgreen}{green!50!black}
\tikzset{dots/.style={very thick,loosely dotted},
         greendot/.style={fill,circle,color=darkgreen,inner sep=1.5pt,outer sep=0}
}
\def\greendot(#1,#2){\node[greendot] at(#1,#2){}}

\newenvironment{braid}{
  \begin{tikzpicture}[baseline=6mm,blue,line width=1pt, scale=0.4,
                      draw/.append style={rounded corners},
                      every node/.append style={font=\fontsize{5}{5}\selectfont}]%
  }{\end{tikzpicture}
}

\def\Grid(#1,#2){
  \draw[very thin,gray,step=2mm] (0,0)grid(#1,#2);
  \draw[very thin,darkgreen,step=10mm] (0,0)grid(#1,#2);
}

\newcommand\Tableau[2][\relax]{
  \begin{tikzpicture}[scale=0.5,draw/.append style={thick,black}]
    \ifx\relax#1\relax%
    \else 
      \foreach\box in {#1} { \filldraw[blue!30]\box+(-.5,-.5)rectangle++(.5,.5); }
    \fi
    \newcount\row\newcount\col
    \row=0
    \foreach \Row in {#2} {
       \col=1
       \foreach\k in \Row {
          \draw(\the\col,\the\row)+(-.5,-.5)rectangle++(.5,.5);
          \draw(\the\col,\the\row)node{\k};
          \global\advance\col by 1
       }
       \global\advance\row by -1
    }
  \end{tikzpicture}
}

\newcommand\YoungDiagram[2][\relax]{
  \begin{tikzpicture}[scale=0.5,draw/.append style={thick,black}]
    \ifx\relax#1\relax%
    \else 
    \foreach\box in {#1} {
      \filldraw[blue!30]\box rectangle ++(1,1);
    }
    \fi
    \newcount\row
    \row=0
    \foreach \col in {#2} {
       \draw(1,\the\row)grid ++(\col,1);
       \global\advance\row by -1
    }
  \end{tikzpicture}
}

\colorlet{darkgreen}{green!50!black}
\tikzset{dots/.style={very thick,loosely dotted},
         greendot/.style={fill,circle,color=darkgreen,inner sep=1.5pt,outer sep=0},
         blackdot/.style={fill,circle,color=black,inner sep=2pt,outer sep=0},
         graydot/.style={fill,circle,color=gray,inner sep=1.1pt,outer sep=0},
         whitedot/.style={fill,circle,color=white,inner sep=3pt,outer sep=0},
}
\def\greendot(#1,#2){\node[greendot] at(#1,#2){}}
\def\blackdot(#1,#2){\node[blackdot] at(#1,#2){}}
\def\graydot(#1,#2){\node[graydot] at(#1,#2){}}
\def\whitedot(#1,#2){\node[whitedot] at(#1,#2){}}

\def\Grid(#1,#2){
  \draw[very thin,gray,step=2mm] (0,0)grid(#1,#2);
  \draw[very thin,darkgreen,step=10mm] (0,0)grid(#1,#2);
}

\theoremstyle{plain} 
\newcommand{\thistheoremname}{}
\newtheorem*{genericthm*}{\thistheoremname}
\newenvironment{namedthm*}[1]
  {\renewcommand{\thistheoremname}{#1}%
   \begin{genericthm*}}
  {\end{genericthm*}}

\begin{document}

\title{Empty runner removal theorem for Ariki-Koike algebras}

\author{A. Dell'Arciprete}
       \address{Department of Mathematics, 
University of York, Heslington, York,  UK}
\email{alice.dellarciprete@york.ac.uk}

\author{L. Putignano}
       \address{Dipartimento di Matematica e Informatica U. Dini, 
       Università degli Studi di Firenze, Firenze, Italy}
\email{lorenzo.putignano@unifi.it}
\begin{abstract}
For the Iwahori-Hecke algebras of type $A$, James and Mathas proved a theorem which relates
$v$-decomposition numbers for different values of $e$, by adding empty runners to the James' abacus display. This result is often referred to as the empty runner removal theorem.
In this paper, we extend this theorem to the Ariki-Koike algebras, establishing a similar relationship for the $v$-decomposition numbers.
\end{abstract}

\maketitle
\section{Introduction}
The Ariki-Koike algebras, also known as cyclotomic Hecke algebras of type $G(r,1,n)$, first appeared in the work of Cherednik \cite{MR899405}. However, Ariki and Koike in ~\cite{AK94} were the first to study them systematically as a simultaneous generalisation of Iwahori-Hecke algebras of type $A$ and type $B$. The interest in these algebras is driven by their connections with many different areas such as group theory \cite{MR1882533}, Khovanov homology, knot theory \cite{MR3709726}, and higher representation theory \cite{MR3732238}.
This ubiquity 
made the representation theory of the Ariki-Koike algebras extensively studied in the past decades. A detailed review of their representation theory is provided in Mathas's paper \cite{Mat04}.

In numerous ways, the Ariki-Koike algebra exhibits similar behaviour to the Iwahori-Hecke algebra  of type $A$, with several combinatorial results being generalised to the Ariki-Koike case, replacing partitions with multipartitions. Let {$\HC$} denote an Ariki-Koike algebra. As for type $A$, there is a class of important $\HC$-modules, the \emph{Specht modules}, which are indexed by the set of $r$-multipartitions of $n$. When $\HC$ is semisimple, the Specht modules form a complete set of non-isomorphic simple $\HC$-modules; otherwise, the simple modules appear as the heads of a specific subset of Specht modules.

One of the main problems in the representation theory of Ariki-Koike algebras is the \textit{decomposition number problem}, which asks for the composition multiplicities of the simple modules in the Specht modules. This turned out to be a very challenging problem.
In principle, the decomposition numbers of the Ariki-Koike algebras can be computed in characteristic zero. However, all known algorithms 
are recursive and in practice it is possible to compute them  only for small values of $n$.
Thus, it would be really helpful to have a way of decreasing the value of $n$ so that we just need to compute decomposition numbers for small $n$.
This is the approach that led the way for a collection of results known under the name of \textit{runner removal theorems}. The name comes from James' abacus display: a configuration of vertical strands with beads, the \emph{runners}, that provides a 
way of representing a partition.

The first runner removal theorem was established for the Iwahori-Hecke algebra of type $A$ and the $q$-Schur algebra by James and Mathas in \cite{JM02}. They relate $v$-decomposition numbers of Iwahori-Hecke algebras of type $A$ for different values of $e$, by adding \emph{empty} runners to the {abacus displays of partitions}. After that, in \cite{Frunrem} Fayers proves a similar theorem, which involves adding \emph{full} runners to {the abacus display}. Fayers' results has been extended to Ariki-Koike algebras in \cite{DellA24b}. The goal of the present paper is to give an empty runner removal theorem for Ariki-Koike algebras. 
This result has been conjectured by Ariki, Lyle and Speyer to be an important step towards the study of Schurian finiteness of Ariki-Koike algebras since James and Mathas' empty runner removal theorem was a key component in their previous work \cite{MR4673430} about Schurian finiteness for blocks of Iwahori-Hecke algebras of type $A$.

Our main result is the following.
\begin{thm}\label{main}
Let $\bm\mu$ be a $e$-multiregular $r$-multipartition of $n$ and let $\bm k\in\Z^r$ such that \eqref{same_runn_d} and \eqref{conditiononk^j} hold. Consider $\bm\mu^{+\bm k}$ the $r$-multipartition constructed as in Section \ref{def+k}. Then, for every multipartition $\bm\la$ of $n$,
$$d_{\bm\lambda\bm\mu}^{\bm s}(v)=d_{\bm\lambda^{+\bm k}\bm\mu^{+\bm k}}^{\bm s^+}(v).$$
\end{thm}

The conditions imposed on $\bm k$ imply that the inserted runners are \emph{empty} according to \cref{def_empty_mlt}; it is then appropriate to talk about \cref{main} as an empty runner removal theorem.
Our methods build on that in \cite{DellA24b} and involve a careful combinatorial study of how Fayers’ LLT-type algorithm behaves when an empty runner is added to the abacus.

The paper is organised as follows.
In Section \ref{sec:basicdef}, we give an overview of any background material that we will need later. This includes both the algebraic setup and several combinatorial definitions. 
In Section \ref{sec:LLTalg}, we introduce the Fock space representation of the quantum group $U_v(\widehat{\mathfrak{sl}}_e)$ and present Fayers' LLT-type algorithm for Ariki-Koike algebras \cite{Fay10}.
In Section \ref{sec:addEmptyRun} we define the addition of empty runners to the abacus display of a multipartition and we describe how this interacts with some relevant generators of $U_v(\widehat{\mathfrak{sl}}_{e+1})$. Finally, in Section \ref{sec:main thm} 
we use the canonical basis vectors of the Fock space corresponding to $\bm\mu$ and $\bm\mu^{+\bm k}$ to prove \cref{main}.

\section*{Acknowledgements}
\textit{The first author was funded by EPSRC grant
EP/V00090X/1. The second author is supported by GNSAGA of INdAM.
We thank Matt Fayers for making his \LaTeX \ style file for abacus displays publicly available and for helpful conversations. We thank Liron Speyer for originally asking the question which led to this project.}

\section{Basic definitions}\label{sec:basicdef}

\subsection{The Ariki-Koike algebras}\label{AKdef}
Let $r\geq 1$ and $n \geq0$. Let $W_{r,n}$ be the complex reflection group $C_r \wr~\mathfrak{S}_n$. 
We define the Ariki-Koike algebra as a deformation of the group algebra $\mathbb{F}W_{r,n}$.
\begin{defn}
Let $\mathbb{F}$ be a field and $q,Q_1,\ldots,Q_r$ be elements of $\mathbb{F}$, with $q$ non-zero. Let $\bm{Q}=(Q_1, \ldots, Q_r)$. The \textbf{Ariki-Koike algebra} $\mathcal{H}_{\mathbb{F},q,\bm{Q}}(W_{r,n})$ of $W_{r,n}$ is defined to be the unital associative $\mathbb{F}$-algebra with generators $T_0, \ldots, T_{n-1}$ and relations
\begin{align*}
(T_0-Q_1)\cdots(T_0-Q_r)&=0,\\
T_0T_1T_0T_1&= T_1T_0T_1T_0,\\
(T_i+1)(T_i-q)&=0, & \text{for }1 &\leq i \leq n - 1,\\
T_iT_j &= T_jT_i, &\text{for }0 &\leq i < j-1 \leq n - 2,\\
T_iT_{i+1}T_i &= T_{i+1}T_iT_{i+1}, &\text{for } 1&\leq  i \leq n - 2.\\
\end{align*}
\end{defn}
Define $e$ to be minimal such that $1 + q + \ldots + q^{e-1}= 0$, or set $e = \infty$ if no such value exists. Throughout this paper we shall assume that $e$ is finite and we shall refer to $e$ as the \textit{quantum characteristic}. Set $I = \{0, 1, \ldots, e-1\}$ which is usually identified  with $\mathbb{Z}/e\mathbb{Z}$. Following \cite{dm02}, we say that the $r$-tuple $\bm{Q}$ is \textit{$q$-connected} if, for each $j\in\{1, \ldots, r\}$, $Q_j = q^{s_j}$ for some $s_j \in I$. Also, Dipper and Mathas prove that any Ariki-Koike algebra is Morita equivalent to a direct sum of tensor products of smaller Ariki-Koike algebras, each of which has $q$-connected parameters. Thus, for the rest of the paper we assume that the Ariki-Koike algebra has $q$-connected parameters with $\bm s=(s_1,\dots,s_r)\in I^r$ and we denote it $\aks=\mathcal{H}_{\mathbb{F},q, \bm{Q}}(W_{r,n})$. We call a \emph{multicharge} for $\aks$ every $r$-tuple of integers $\bm a=(a_1,\dots,a_r)$ such that $a_j\equiv s_j\  (\mod e)$ for every $1\le j\le r$.

\subsection{Multipartitions}\label{multpar}
Let $n\ge0$. A \emph{partition} of $n$ is defined to be a non-increasing sequence $\lambda = (\lambda_1, \lambda_2, \dots)$ of non-negative integers whose sum is $n$. {The integers $\lambda_b$, for $b\geq1$, are called the \emph{parts} of $\lambda$. We write $|\lambda| = n$. For any partition $\la$, we write $l(\lambda)$ for the number of non-zero parts of $\la$,
which we call the \emph{length} of $\lambda$. 
We write $\varnothing$ for the unique partition of $0$. When writing a partition, we group together equal parts with a superscript and we omit trailing zeros. For example, $$(4,4,2,1,0,0, \dots) = (4,4,2,1) = (4^2,2,1).$$
The \emph{Young diagram} of a partition $\lambda$ is the set
$$[\lambda] := \{(b,c) \in \mathbb{Z}_{>0} \times \mathbb{Z}_{>0
}\text{ } | \text{ } c\leq \lambda_b\}.$$ The elements of $[\lambda]$, and more generally of $(\mathbb{Z}_{>0})^2$, are called \emph{nodes}.
Let $2\leq e<\infty$ and $a\ge0$. We define the \emph{$e$-residue} of a node $(b,c)$ with respect to $a$ to be
$$\mathrm{res}_{e,a}(b,c) = c-b+a\  (\mod  e).$$
For each $l\geq 1$, we define the $l^{\text{th}}$ \emph{ladder} to be the set
$$\mathcal{L}_l = \{(b,c) \in (\mathbb{Z}_{>0})^2 \text{ }| \text{ }b + (e - 1)(c - 1) = l\}.$$
Fixed $a\ge0$, all the nodes in $\mathcal{L}_l$ have the same residue (namely, $ a+1 - l\ (\mod e)$), so we define the residue of $\mathcal{L}_l$ (with respect to $a$) to be
this residue. If $\la$ is a partition, the $l^{\text{th}}$ ladder $\mathcal{L}_l(\la)$ of $\la$ is the intersection of $\mathcal{L}_l$ with the Young diagram of $\la$.

\begin{exe} Suppose $e = 3$, and $\la = (5,3,2^2,1)$. Consider the Young diagram of $\la$. Then in the first diagram we label each node of $[\la]$ with the number of the ladder in which it lies, while in the second one we fill the nodes with their residues with respect to $a=2$:

\Yvcentermath1
$$\young(13579,246,35,46,5) \quad \quad \quad \young(20120,120,01,20,1).$$
\end{exe}

\noindent We also recall the following definitions for partitions.
\begin{itemize}
    \item A partition $\lambda$ is \emph{$e$-singular} if $\lambda_{i+1} = \lambda_{i+2} = \ldots = \lambda_{i+e} > 0$ for some $i$. Otherwise, $\lambda$ is \emph{$e$-regular}. For example, for $e=3$ the partition $(4, 3, 1)$ is $3$-regular, while the partition $(4,1^4)$ is $3$-singular. 
    \item A partition $\lambda$ is \emph{$e$-restricted} if $\lambda_i -\lambda_{i+1} < e$ for every $i \geq1$.
\end{itemize}

\begin{defn}\label{multipar}
An $r$-\emph{multipartition} of $n$ is an ordered $r$-tuple $\bm{\lambda} = (\lambda^{(1)}, \dots, \lambda^{(r)})$  of partitions such that $$|\bm{\lambda}| := |\lambda^{(1)}|+ \ldots +|\lambda^{(r)}| = n.$$
If $r$ is understood, we shall just call this a multipartition of $n$.
We write $\mathcal{P}^r$ for the set of $r$-multipartitions.
\end{defn}

We write $\bm\varnothing$ for the unique multipartition of $0$. The \emph{Young diagram} of a multipartition $\bm\lambda$ is the set
$$
[\bm{\lambda}]:=\{(b,c,j)\in \mathbb{Z}_{>0}\times \mathbb{Z}_{>0}\times\{1, \ldots, r\} \text{ }|\text{ } c \leq \lambda_b^{(j)}\}.
$$
We may abuse notation by not distinguishing a multipartition from its Young diagram.}
For multipartitions, a node means an element of $(\mathbb{Z}_{>0})^2\times\{1, \ldots, r\}$. We say that a node $\mathfrak{n} \in [\bm{\lambda}]$ is \emph{removable} if $[\bm{\lambda}]\setminus \{\mathfrak{n}\}$ is the Young diagram of a multipartition, and we say that a node $\mathfrak{n} \notin [\bm{\lambda}]$ is \emph{addable} if $[\bm{\lambda}] \cup \{\mathfrak{n}\}$ is the Young diagram of a multipartition. 


{Given $\bm{a}=(a_1, \ldots, a_r)\in\mathbb{Z}^r$, to each node $(b,c,j) \in [\bm\lambda]$ we associate its $e$-\emph{residue} with respect to $\bm a$ \[\mathrm{res}_{e,\bm{a}}(b,c,j) = a_j + c-b\ \ (\mod e).\] If $i\in\zez$, a node of residue $i$ is called an $i$-node. We denote by $\rem{i}{\bla}$ and $\add{i}{\bla}$ the number of removable and addable $i$-nodes of $\bla$, respectively.

{
\subsection{Regular multipartitions}\label{ssec:Kl_mpt}
Fixed a multicharge $\bm{a}$ for $\aks$, we consider the residues of nodes with respect to $\bm{a}$. This will be our convention throughout the paper. Residues are useful in classifying the simple $\aks$-modules. Indeed, the notion of residue helps us to describe a certain subset of $\mathcal{P}^r$, which index the simple modules for $\aks$.
We impose a partial order $\succ$ on the set of nodes of $e$-residue $i \in I$ of a multipartition by saying that $(b,c,j)$ is \emph{below} $(b',c',j')$ (equivalently that $(b',c',j')$ is \emph{above} $(b,c,j)$) if either $j>j'$ or ($j=j'$ and $b>b'$). In this case we write $(b,c,j)\succ(b',c',j')$. Note this order restricts to a total order on the set of all addable and removable nodes of residue $i \in I$ of a multipartition. Given a multipartition $\bla$ and a node $\mathfrak{n}\in[\bm\la]$, we write $\remab{i}{\bla}$ and $\addab{i}{\bla}$ for the number of removable and addable $i$-nodes of $\bla$ above $\mathfrak{n}$, respectively.

Suppose $\bla$ is a multipartition, and given $i \in I$ define the \emph{$i$-signature of $\bla$ with respect to $\succ$} by examining all the addable and removable $i$-nodes of $\bla$ in turn from lower to higher, and writing a $+$ for each addable $i$-node and a $-$ for each removable $i$-node. Now construct the \emph{reduced $i$-signature with respect to $\succ$} by successively deleting all adjacent pairs $-+$. If there are any $-$ signs in the reduced $i$-signature of $\bm\lambda$, the lowest of these nodes is called the \emph{good $i$-node of $\bla$ with respect to $\succ$}.

\begin{defn}
We say that $\bm\la$ is a \emph{regular multipartition} if and only if there is a sequence
$$\bm\lambda = \bm\lambda(n), \bm\lambda(n-1), \ldots, \bm\lambda(0)= \bm\varnothing$$
of multipartitions such that for each $k$, $[\bm\lambda(k-1)]$ is obtained from $[\bm\lambda(k)]$ by removing a good node with respect to $\succ$.
\end{defn}

This definition depends on the $q$-connected parameters $\bm s$ of $\aks$. We write $\mathcal{R}'(\bm s)$ for the set of regular multipartitions with respect to $\bm s$.
Note that the regular multipartitions are the dual version of the multipartitions called \emph{Kleshchev} or \emph{restricted} in \cite{BK09}. There is a canonical bijection between restricted and regular multipartitions of a fixed integer; for details see \cite{mythesis, BK09}. 

\begin{exe}\label{Kl_exe}
Suppose $r=2$ and $e=4$. Consider $\bm s =(1,0)\in(\Z/4\Z)^2$. Then the multipartition $\bm\la=((3,1), (1^2))$ is regular. Indeed, we have the following sequence of multipartitions obtained from $\bm\la$ by removing each time a good node (we write $[\bm\la(k)]\xleftarrow{i} [\bm\la(k-1)]$ to denote that $[\bm\la(k-1)]$ is obtained from $[\bm\la(k)]$ by removing the good $i$-node for $i\in I$):
$$\Yinternals1\Yaddables1 \begin{matrix*}[l]\Ycorner1\yngres(4,3,1) \\ \\ \Ycorner0\yngres(4,1^2) \end{matrix*} \xleftarrow{3} \begin{matrix*}[l]\Ycorner1\yngres(4,2,1) \\ \\ \Ycorner0\yngres(4,1^2) \end{matrix*} \xleftarrow{0} \begin{matrix*}[l]\Ycorner1\yngres(4,2) \\ \\ \Ycorner0\yngres(4,1^2) \end{matrix*} \xleftarrow{2} \begin{matrix}\Ycorner1\yngres(4,1) \\ \\ \Ycorner0\yngres(4,1^2) \end{matrix}
\xleftarrow{3}\begin{matrix}\Ycorner1\yngres(4,1) \\ \\ \Ycorner0\yngres(4,1) \end{matrix}
\xleftarrow{1}\begin{matrix*}[l] 1 \\ \\ \Ycorner0\yngres(4,1) \end{matrix*}
\xleftarrow{0}
\begin{matrix}\varnothing \\ \\ \varnothing \end{matrix}.$$
\end{exe}

\begin{rmk}\cite{BK09}\label{kl=restr}
If $r = 1$, then the multipartition $(\la)$ is regular if and only if $\la$ is an $e$-regular partition.
\end{rmk}

\begin{defn}
We say that a multipartition $\bm\la$ is \emph{$e$-multiregular} if all its components are $e$-regular partitions.  We write $\mathcal{R}$ for the set of $e$-regular partitions and $\mathcal{R}^r$ for the set of $e$-multiregular $r$-multipartitions.
\end{defn}

\begin{rmk}
Note that $\mathcal{R}'(\bm s) \subseteq \mathcal{R}^r$. This is a dual version of \cite[Proposition 4.8]{MR1603195}.
\end{rmk}

We introduce the following notation for multipartitions, that will be useful in \cref{sec:LLTalg}. If $\bm\la=(\la^{(1)},\dots,\la^{(r)})$ is an $r$-multipartition for $r>1$, we write $\bm\la_-$ for the ($r-1$)-multipartition $(\la^{(2)},\dots,\la^{(r)})$ and $\bm\la_0$ for the $r$-multipartition $(\varnothing,\bm\la_-)=(\varnothing,\la^{(2)},\dots,\la^{(r)})$.  

%

\subsection{Specht module and simple modules}\label{specht}
The algebra $\aks$ is a cellular algebra \cite{DJM98,GL96} with the cell modules indexed by $r$-multipartitions of $n$. For each $r$-multipartition $\bm\lambda$ of $n$ we define a $\aks$-module $S'(\bm\lambda)$ called \emph{(dual) Specht module}; these modules are the cell modules defined in \cite[Chapter 4]{Mat03}. 
When $\aks$ is semisimple, the dual Specht modules form a complete set of pairwise non-isomorphic simple $\aks$-modules. However, we are mainly interested in the case when the algebra is not semisimple. In this case, we set $D'(\bm\la) = S'(\bm\la)/ \mathrm{rad}\text{ } S'(\bm\la)$ where $\mathrm{rad}\text{ } S'(\bm\la)$ denotes the radical of the bilinear form of $ S'(\bm\la)$ defined in terms of the structure constants of the cellular basis of $\aks$. Then, as shown in \cite[Theorem 4.2]{A01}, a complete set of pairwise non-isomorphic simple $\aks$-modules is given by
$$\{D'(\bm\la)\text{ }|\text{ } \bm\la\in\mathcal{R}'(\bm s)\}.$$

}

If $\bm\lambda$ and $\bm\mu$ are $r$-multipartition of $n$ with $\bm\mu$ regular, let $[S'(\bm\lambda):D'(\bm\mu)]$ denote the multiplicity of the simple module $D'(\bm\mu)$ as a composition factor of the Specht module $S'(\bm\lambda)$. The matrix $D = ([S'(\bm\lambda):D'(\bm\mu)])_{\bm\la, \bm\mu}$ is called the \emph{decomposition matrix} of $\aks$ and determining its entries is one of the most outstanding problems in the representation theory of Ariki-Koike algebras. It follows from the cellularity of $\aks$ that the decomposition matrix is unitriangular; to state this we need to define a partial order on multipartitions. Given two multipartitions $\bm{\lambda}$ and $\bm{\mu}$ of $n$, we say that $\bm{\lambda}$ \emph{dominates} $\bm{\mu}$, and write $\bm{\lambda} \trianglerighteq \bm{\mu}$, if {
$$\sum_{a=1}^{j-1}|\lambda^{(a)}| + \sum_{b=1}^{i} \lambda_b^{(j)} \geq \sum_{a=1}^{j-1}|\mu^{(a)}| + \sum_{b=1}^{i}\mu_b^{(j)}$$
for $j=1,2, \ldots, r$ and for all $i \geq 1$}. We refer to this order as the dominance order on multipartitions.

\begin{thm}\cite{DJM98,GL96}\label{decmtxHQ}
Let $\bm\lambda$ and $\bm\mu$ be $r$-multipartitions of $n$ with $\bm\mu$ regular.
\begin{enumerate}
\item If $\bm\mu=\bm\lambda$, then $[S'(\bm{\lambda}):D'(\bm{\mu})]=1$.
\item If $[S'(\bm{\lambda}):D'(\bm{\mu})]>0$, then $\bm{\lambda} \trianglelefteq \bm{\mu}$.
\end{enumerate}
\end{thm}

\subsection{The abacus}\label{betaAbacus}
Given the assumption that the Ariki-Koike algebra $\aks$ has $q$-connected parameters, we may conveniently represent multipartitions on an abacus display.

Fix $e\ge2$ and take an abacus with $e$ vertical runners labelled by the symbols $0,\dots,e-1$ from left to right. We index the \emph{positions} in the $i^{\text{th}}$ runner by the integers $i,i+e,i+2e,\dots$ from top down and we align runners so that position $x$ is immediately to the right of position $x-1$ whenever $e\nmid x$. We call \emph{level} each row of the abacus, e.g. level 0 is the topmost row and comprises positions from 0 to $e-1$. Each position in the abacus is either a \emph{bead} $\abacus(b)$ or a \emph{space} $\abacus(n)$. A space is sometimes called an \emph{empty position}.

Consider a partition $\la$ and an integer $a\ge l(\la)$. We construct the \emph{$e$-abacus display of $\la$ with charge $a$}, denoted $\mathrm{Ab}_{e}^{a}(\la)$, putting a bead at position $\la_j-j+a$ for every $1\le j\le a$. Hence the charge gives the number of beads in the abacus display. Throughout the paper we may use the word \emph{configuration} in place of display.

Note that, by construction, every runner in the abacus display of a partition has infinitely many spaces downwards. When we draw an abacus, we use the convention that all positions below those shown are spaces. Also, observe that increasing the charge by $e$ entails adding a row of beads at the top of the abacus display.

Let $\bm{\lambda}=(\lambda^{(1)}, \ldots, \lambda^{(r)})$ be a $r$-multipartition of $n$ and fix $\bm{a}=(a_1, \ldots, a_r)\in \mathbb{Z}^{r}$
a multicharge for $\aks$ such that $a_j\ge l(\lambda^{(j)})$ for every $1\le j\le r$. The \emph{$e$-abacus display of $\bm\la$ with multicharge $\bm a$}, denoted $\mathrm{Ab}_e^{\bm a}(\bla)$ is the $r$-tuple of $e$-abacus displays associated to the components of $\bla$. We use the term \emph{multirunner} to mean the family of runners having the same label in the abacus display of a multipartition.

\begin{exe}
Let $e=4$. Suppose that $r = 3$, $\bla= ((6^3,4,2^4,1), \varnothing, (9,7^5,3,1^4))$ and $\bm s=(0,1,1)$.
The 4-abacus display of $\bla$ with multicharge $\bm{a}=(12,5,13)$ is as follows.

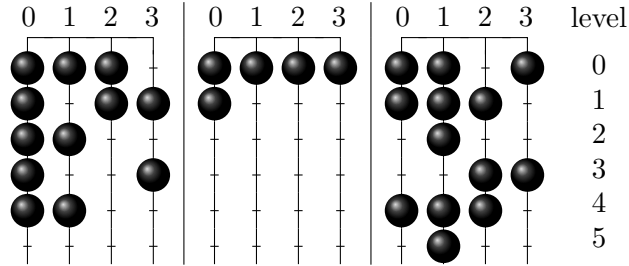
\begin{center}
\begin{tabular}{c|c|cc}
$\begin{matrix} 0&1&2&3
\end{matrix}$ & $\begin{matrix} 0&1&2&3
\end{matrix}$ & $\begin{matrix} 0&1&2&3
\end{matrix}$ & \text{level}\\
\sabacus(1.5,lmmr,bbbn,bnbb,bbnn,bnnb,bbnn,nnnn) & \sabacus(1.5,lmmr,bbbb,bnnn,nnnn,nnnn,nnnn,nnnn) & \sabacus(1.5,lmmr,bbnb,bbbn,nbnn,nnbb,bbbn,nbnn) & $\begin{matrix}
0\\
1\\
2\\
3\\
4\\
5
\end{matrix}$
\end{tabular}
\end{center}
\end{exe}

\subsection{Rim $e$-hooks and $e$-core}
We recall some other useful definitions about the Young diagram of a partition. 

\begin{defn}
Suppose $\lambda$ is a partition and $(i,j)$ is a node of $[\lambda]$.
\begin{enumerate}
\item The \emph{rim} of $[\lambda]$ is the set of nodes
$$\{(b,c) \in [\lambda] \text{ }|\text{ }(b+1,c+1) \notin [\lambda]\}.$$
\item A \emph{$e$-rim hook} is a connected subset $R$ of the rim of $\la$ containing $e$ nodes and such that $[\lambda]\setminus R$ is the diagram of a partition.
\item  A \emph{$e$-core} is a partition with no $e$-rim hooks.
\end{enumerate}
\end{defn}


The abacus display is useful for visualising the removal of $e$-rim hooks. If we are given an abacus display for $\lambda$, then $[\lambda]$ has a $e$-rim hook if and only if there is a bead that has a space immediately above. Furthermore, removing a $e$-rim hook corresponds to sliding a bead up one position within its runner. So, a partition is a $e$-core if and only if every bead in the abacus display has a bead immediately above it. In particular, this implies that the $e$-core of a partition is well defined.
%

Finally, we can notice that each bead corresponds to a row of the diagram of $\lambda$ (or to a row of length $0$), and hence to a part of $\la$. The latter can be detected from the abacus counting the spaces that precedes the corresponding bead. If the node at the end of a row of $\la$ (if it exists) has residue $i\ (\mod e)$ with respect to the charge of the abacus, for some $i \in I$, then the corresponding bead is on runner $i$. Thus, moving a bead from runner $i$ to runner $i+1\ (\mod e)$ is equivalent to adding a node of residue $i + 1\ (\mod e)$ to the diagram of $\la$. Similarly, moving a bead from runner $i$ to runner $i - 1\ (\mod e)$ is equivalent to removing a $i$-node from the diagram of $\la$.

\section{An LLT-type algorithm for Ariki-Koike algebras}\label{sec:LLTalg}

In this section we consider the integrable representation theory of the quantised enveloping algebra $\UU=U_v(\widehat{\mathfrak{sl}}_e)$, defined in \cref{ssec:qtm_alg}. For any dominant integral weight $\Lambda$ for $\UU$, the simple highest-weight module $V(\Lambda)$ for $\UU$ can be constructed as a submodule $M^{\bm s}$ of a \emph{Fock space} $\mathcal{F}^{\bm s}$ (which depends not just on $\Lambda$ but on an ordering of the fundamental weights involved in $\Lambda$). Using the standard basis of the Fock space, one can define a \emph{canonical basis} 
 for $M^{\bm s}$. There is considerable interest in computing this canonical basis, that is, computing the transition coefficients from the canonical basis to the standard basis. This is because of Ariki's theorem, which says that the transition coefficients evaluated at $v=1$ give the decomposition numbers for Ariki-Koike algebras.  The LLT algorithm due to Lascoux, Leclerc and Thibon \cite{LLT96} computes the canonical basis whenever $\Lambda$ is a weight of level $1$. The purpose of this section is to present the generalisation of this algorithm to higher levels given by Fayers in \cite{Fay10}.
Fayers' algorithm computes the canonical basis for an intermediate module $M^{\otimes\bm s}$, which is defined to be the tensor product of level $1$ highest-weight simple modules.  It is then straightforward to discard unwanted vectors to get the canonical basis for $M^{\bm s}$. In order to describe Fayers' algorithm we need to introduce some notation as we do in the next subsection.

%
%

\subsection{The quantum algebra $U_v(\widehat{\mathfrak{sl}}_e)$ and the Fock space}\label{ssec:qtm_alg}

Denote with $\UU$ the quantised enveloping algebra $U_v(\widehat{\mathfrak{sl}}_e)$. This is a $\mathbb{Q}(v)$-algebra with generators $e_i,f_i$ for $i\in I$ and $v^h$ for $h\in P^\vee$, where $P^\vee$ is a free $\mathbb{Z}$-module with basis $\{h_i\mid i\in I\}\cup\{d\}$. 
These generators are subjected to well known relations which can be found in \cite[\S 4.1]{LLT96}.
Here we follow the usual notation for $v$-integers, $v$-factorials and $v$-binomial coefficients:
$$[k]=\dfrac{v^{k}-v^{-k}}{v-v^{-1}}, \quad [k]!=[k][k-1] \cdots [1], \quad \left[\begin{matrix}
m\\k
\end{matrix}\right]=\dfrac{[m]!}{[m-k]![k]!}.$$
For any integer $m>0$, we write $f_i^{(m)}$ to denote the quantum divided power $f_i^m/[m]!$.
Since $\UU$ is a Hopf algebra with comultiplication denoted $\Delta$ in \cite{Kas02}, the tensor product of two $\UU$-modules can be regarded as a $\UU$-module.
The $\mathbb{Q}$-linear ring automorphism $\overline{\phantom{o}}:\UU\to\UU$ defined by
\[\overline{e_i}=e_i,\qquad \overline{f_i}=f_i,\qquad \overline v=v^{-1},\qquad \overline{v^h}=v^{-h}\]
for $i\in I$ and $h\in P^\vee$ is called \emph{bar involution}.

Now we fix $\bm s\in I^r$ for some $r\geq1$, and define the \emph{Fock space} $\mathcal{F}^{\bm s}$ to be the $\mathbb{Q}(v)$-vector space with basis $\{\bm\la \mid \bm\la\in\mathcal{P}^r\}$, which we call the \emph{standard basis}.  The Fock space has the structure of a $\UU$-module: for a full description of the module action, we refer to \cite{Fay10}. Here, we describe the action of the generators $f_0, \ldots, f_{e-1}$.

Given $\bm\la, \bm\xi \in\mathcal{P}^r$, we write $\bm\la\xrightarrow{m:i}\bm\xi$ to indicate that $\bm\xi$ is obtained from $\bm\la$ by adding $m$ addable $i$-nodes. Equivalently, this means that an abacus display for $\bm\xi$ is obtained from that of $\bla$ by moving $m$ beads from multirunner $i-1$ to multirunner $i$. If this is the case, then we consider the total order $\succ$ on addable and removable nodes and we define the integer

\[N_i(\bm\la,\bm\xi)= \sum_{\mathfrak{n}\in\bm\xi\setminus\bm\la}\addab{i}{\bm\xi}-\remab{i}{\bla}.\]
Now the action of $f_i^{(m)}$ is given by
$$f_i^{(m)}\bm\la = \sum_{\bm\la\xrightarrow{m:i}\bm\xi}v^{N_i(\bm\la,\bm\xi)}\bm\xi.$$
We recall the following result about $N_i(\bm\la, \bm\xi)$.
\begin{prop}\cite[Proposition 3.2]{DellA24b}\label{Ncomps}
Let $i\in I$. Suppose $\bm\la$ and $\bm\xi$ are $r$-multipartitions such that $\bm\la\xrightarrow{m:i}\bm\xi$. Then
$$N_i(\bla,\bm\xi)= \sum_{\mathfrak{n}\in \bm\xi\setminus\bm\la}(\addab{i}{\xi^{(J_{\mathfrak{n}})}}-\remab{i}{\la^{(J_{\mathfrak{n}})}})+\sum_{j=1}^{J_{\mathfrak{n}}-1}\add{i}{\xi^{(j)}}-\rem{i}{\la^{(j)}},$$
where $J_{\mathfrak{n}}$ is the component of the node $\mathfrak{n}$ in $\bm\xi$. 
\end{prop}

The reason why we are interested in the Fock space is because the submodule $M^{\bm s}$ generated by the \emph{empty} multipartition $\bm \varnothing=(\varnothing, \ldots, \varnothing)$ is isomorphic to the simple highest-weight module $V(\Lambda_{s_1}+\dots+\Lambda_{s_r})$. This submodule inherits a bar involution from $\UU$: this is defined by $\overline{\bm\varnothing}=\bm\varnothing$ and $\overline{um} = \overline u\,\overline m$ for all $u\in\UU$ and $m\in M^{\bm s}$.  The bar involution allows one to define a \emph{canonical basis} for $M^{\bm s}$; this consists of vectors $G^{\bm s}_e(\bm\mu)$, for $\bm\mu$ lying in some subset of $\mathcal{P}^r$ (with our conventions, this is the set of regular multipartitions). These canonical basis vectors are characterised by the following properties:
\begin{itemize}
\item
$\overline{G^{\bm s}_e(\bm\mu)}=G^{\bm s}_e(\bm\mu)$;
\item
if we write $G^{\bm s}_e = \sum_{\bm\la\in\mathcal{P}^r}d^{\bm s}_{\bm\la\bm\mu}(v)\bm\la$ with $d^{\bm s}_{\bm\la\bm\mu}(v)\in\mathbb{Q}(v)$, then $d^{\bm s}_{\bm\mu\bm\mu}(v)=1$, while $d^{\bm s}_{\bm\la\bm\mu}(v)\in v\mathbb{Z}[v]$ if $\bm\la\neq\bm\mu$; in particular, $d^{\bm s}_{\bm\la\bm\mu}(v)= 0$ unless $\bm \mu \trianglerighteq \bm\la$.
\end{itemize}

{A lot of effort has been put in computing the canonical basis elements (i.e. computing the transition coefficients $d^{\bm s}_{\bm\la\bm\mu}(v)$), because of the following theorem. 

\begin{thm}\cite[Theorem 4.4]{A96}\label{arithm}
Let $\mathbb{F}$ be a field of characteristic $0$ and $\bm s \in I^r$. Suppose $\bm\la$, $\bm\mu$ are $r$-multipartitions of $n$ with $\bm\mu$ regular. 
Then
$$[S'(\bm\la)\colon D'(\bm\mu)]=d_{\bm\la\bm\mu}^{\bm s}(1).$$
\end{thm}

In fact, this theorem says that the coefficients $d^{\bm s}_{\bm\la\bm\mu}(v)$ specialised at $v = 1$ give the decomposition numbers of the Ariki-Koike algebras $\aks$ over a characteristic zero field. More generally, Brundan and Kleshchev \cite{BK09} show that the coefficients $d^{\bm s}_{\bm\la\bm\mu}(v)$ (with $v$ still indeterminate) can be regarded as \textit{graded decomposition numbers} where `graded' refers to the $\Z$-grading that Ariki-Koike algebras inherit by being isomorphic to certain quotients of KLR algebras. 
}

In \cite{Y_can-basis}, Yvonne extends the bar involution on $M^{\bm s}$ to the whole of $\mathcal{F}^{\bm s}$. 
The extension yields a canonical basis for the whole of $\mathcal{F}^{\bm s}$ indexed by the set $\mathcal{P}^r$ of $r$-multipartitions. In particular we get the following result.
\begin{thm}\label{decdom}\cite{Fay10}
For each multipartition $\bm\mu$, there is a unique vector
$$G^{\bm s}_e(\bm \mu) = \sum_{\bm\la\in\mathcal{P}^r}d^{\bm s}_{\bm\la\bm\mu}(v)\bm\la \in \mathcal{F}^{\bm s} \text{ with } d^{\bm s}_{\bm\la\bm\mu}(v)\in\mathbb{Q}(v)$$
such that
\begin{itemize}
\item
$\overline{G^{\bm s}_e(\bm\mu)}=G^{\bm s}_e(\bm\mu)$;
\item
$d^{\bm s}_{\bm\mu\bm\mu}(v)=1$, while $d^{\bm s}_{\bm\la\bm\mu}(v)\in v\mathbb{Z}[v]$ if $\bm\la\neq\bm\mu$;
\item
$d^{\bm s}_{\bm\la\bm\mu}(v)= 0$ unless $\bm \mu \trianglerighteq \bm\la$.
\end{itemize}
\end{thm}

In principle, the extension of the bar involution gives an algorithm for computing the canonical basis of $M^{\bm s}$. However, we use an algorithm by Fayers \cite{Fay10} that generalises the LLT algorithm (see \cite{LLT96} for more details and examples in this case) and works better for the purposes of this paper. We recall it in the next subsection.

\subsection{An LLT-type algorithm for $r\geq 1$}\label{algsec}

Fayers' algorithm approach for computing the canonical basis of $M^{\bm s}$ relies on computing the canonical basis for a module lying in between $M^{\bm s}$ and $\mathcal{F}^{\bm s}$. The way $\mathcal{F}^{\bm s}$ is defined and the choice of a coproduct on $\UU$ mean that there is an isomorphism
\begin{alignat*}2
\mathcal{F}^{\bm s}&\overset{\sim}\longrightarrow\,\,&\mathcal{F}^{(s_1)}&\otimes\dots\otimes\mathcal{F}^{(s_r)}\\
\intertext{defined by linear extension of}
\bm\la&\longmapsto&(\la^{(1)})&\otimes\dots\otimes(\la^{(r)}).
\end{alignat*}
Thus, we identify $\mathcal{F}^{\bm s}$ and $\mathcal{F}^{(s_1)}\otimes\dots\otimes\mathcal{F}^{(s_r)}$ via this isomorphism. Since each $\mathcal{F}^{(s_k)}$ contains a submodule $M^{(s_k)}$ isomorphic to $V(\Lambda_{s_k})$, $\mathcal{F}^{\bm s}$ contains a submodule $M^{\otimes\bm s}=M^{(s_1)}\otimes\dots\otimes M^{(s_r)}$ isomorphic to $V(\Lambda_{s_1})~\otimes~\dots~\otimes~V(\Lambda_{s_r})$.  Fayers' algorithm computes the canonical basis of $M^{\otimes\bm s}$.

Before going on, we need the following result on canonical basis coefficients that is crucial in order to apply one of the steps of the algorithm. Recall that for any $r$-multipartition $\bm\la$ we define $\bm\la_{-}=(\la^{(2)}, \ldots, \la^{(r)})$; we also define $\bm s_- = (s_2, \ldots, s_r)\in I^{r-1}$.

\begin{prop}\label{fcomp}\cite[Corollary 3.2]{Fay10}
Suppose $\bm s\in I^r$ for $r>1$ and $\bm\mu\in\mathcal{P}^r$ with $\mu^{(1)}=\varnothing$.  If we write
\begin{align*}
G^{\bm s_-}_e(\bm\mu_-) &= \sum_{\bm\nu\in\mathcal{P}^{r-1}}d^{\bm s_-}_{\bm\nu\bm\mu_-}(v)\bm\nu,\\
\intertext{then}
G^{\bm s}_e(\bm\mu) &= \sum_{\bm\nu\in\mathcal{P}^{r-1}}d^{\bm s_-}_{\bm\nu\bm\mu_-}(v)(\varnothing,\bm\nu).
\end{align*}
\end{prop}

The canonical basis elements $G^{(s_k)}_e(\mu)$ indexed by $e$-regular partitions $\mu$ form a basis for $M^{(s_k)}$, hence the tensor product $M^{(s_1)}\otimes\dots\otimes M^{(s_r)}$ has a basis consisting of all vectors $G^{(s_1)}_e(\mu^{(1)})\otimes\dots\otimes G^{(s_r)}_e(\mu^{(r)})$, where $\mu^{(1)},\dots,\mu^{(r)}$ are $e$-regular partitions. Translating this to the Fock space $\mathcal{F}^{\bm s}$, we find that $M^{\otimes \bm s}$ has a basis consisting of vectors
\[H^{\bm s}_e(\bm\mu) = \sum_{\bm\la\in\mathcal{P}^r}d^{(s_1)}_{\la^{(1)}\mu^{(1)}}(v)\dots d^{(s_r)}_{\la^{(r)}\mu^{(r)}}(v)\bm\la\]
for all $e$-multiregular multipartitions $\bm\mu$.  In fact, we have the following.
\begin{prop}\cite[Proposition 4.2]{Fay07}\label{algo}
The canonical basis vectors $G^{\bm s}_e(\bm\mu)$ indexed by $e$-multiregular $r$-multipartitions $\bm\mu$ form a basis for the module $M^{\otimes\bm s}$.
\end{prop}
In particular, \cref{algo} implies that the span of the canonical basis vectors is a $\mathcal{U}$-submodule of $\mathcal{F}^{\bm s}$. This enables the algorithm to work recursively constructing canonical basis vectors labelled by $e$-multiregular multipartitions.
As in the LLT algorithm for $r=1$, in order to construct the canonical basis vector $G^{\bm s}_e(\bm\mu)$, we construct an auxiliary vector $A(\bm\mu)$ which is bar-invariant, and which we know equals $G^{\bm s}_e(\bm\mu)$ plus a linear combination of `lower' canonical basis vectors; the bar-invariance of $A(\bm\mu)$, together with dominance properties, allows these lower terms to be stripped off. Moreover, we have that $A(\bm\mu)$ lies in $M^{\otimes\bm s}$; then \cref{algo} tells that all the canonical basis vectors occurring in $A(\bm\mu)$ are labelled by $e$-multiregular multipartitions, and therefore we can assume that these have already been constructed.

In fact, the proof of Proposition \ref{algo}, together with the LLT algorithm for partitions, gives us an LLT-type algorithm for Ariki-Koike algebras. We formalise this as follows.

First we define a partial order on multipartitions which is finer than the dominance order: write $\bm\mu\succcurlyeq\bm\nu$ if either $|\mu^{(1)}|>|\nu^{(1)}|$ or $\mu^{(1)}\trianglerighteq\nu^{(1)}$.  Next we describe the steps of our recursive algorithm on the partial order $\succcurlyeq$. If $r>1$, when computing $G^{\bm s}_e(\bm\mu)$ for $\bm\mu\in\mathcal{R}^r$, we assume that we have already computed the vector $G^{\bm s_-}_e(\bm\mu_-)$ and the vectors $G^{\bm s}_e(\bm\nu)$ for all $\bm\nu\in\mathcal{R}^r$ with $\bm\mu\succ\bm\nu$.

\begin{enumerate}
\item[(1)]
If $\bm\mu=\bm\varnothing$,  then $G^{\bm s}_e(\bm\mu)=\bm\varnothing$.
\item[(2)]
If $\bm\mu\neq\bm\varnothing$ but $\mu^{(1)}=\varnothing$, then compute the canonical basis vector $G^{\bm s_-}_e(\bm\mu_-)$. Then,  by \cref{fcomp}, $G^{\bm s}_e(\bm\mu)$ is given by
$$G^{\bm s}_e(\bm\mu) = \sum_{\bm\nu\in\mathcal{P}^{r-1}}d^{\bm s_-}_{\bm\nu\bm\mu_-}(v)(\varnothing,\bm\nu).$$
\item[(3)]
If $\mu^{(1)}\neq\varnothing$, then apply the following procedure.
\begin{enumerate}
\item
Let $\bm\mu_0=(\varnothing,\mu^{(2)},\dots,\mu^{(r)})$, note that $\bm\mu\succ\bm\mu_0$ and compute $G^{\bm s}_e(\bm\mu_0)$.
\item
Let $m_1,\ldots,m_t$ be the sizes of the non-empty ladders of $\mu^{(1)}$ in increasing order, and $i_1,\ldots,i_t$ be their $e$-residues.  Define $A(\bm \mu) = f_{i_t}^{(m_t)}\ldots f_{i_1}^{(m_1)}G^{\bm s}_e(\bm\mu_0)$ and write $A(\bm\mu)=\sum_{\bm\nu\in\mathcal{P}^r}a_{\bm\nu}\bm\nu$.  
\item\label{stepc}
If there is no $\bm\nu\neq\bm\mu$ for which $a_{\bm\nu}\notin v\mathbb{Z}[v]$, then stop.  Otherwise, take such a $\bm\nu$ which is maximal with respect to the dominance order, let $\alpha$ be the unique element of $\mathbb{Z}[v+v^{-1}]$ for which $a_{\bm\nu}-\alpha\in v\mathbb{Z}[v]$, replace $A(\bm\mu)$ with $A(\bm\mu)-\alpha G^{\bm s}_e(\bm\nu)$, and repeat. The vector that remains at the end of the process is $G^{\bm s}_e(\bm\mu)$.
\end{enumerate}
\end{enumerate}

The vector $A(\bm\mu)$ computed in step 3(b) is a bar-invariant element of $M^{\bm s}$, because $G^{\bm s}_e(\bm\mu_0)$ is. Hence by \cref{algo}, $A(\bm\mu)$ is a $\mathbb{Q}(v+v^{-1})$-linear combination of canonical basis vectors $G^{\bm s}_e(\bm\nu)$ with $\bm\nu\in\mathcal{R}^r$.  Furthermore, the action of $f_i$ on a multipartition and the combinatorial results used in the LLT algorithm imply that $a_{\bm\mu}=1$, and that if $a_{\bm\la}\neq0$, then $\bm\mu\succcurlyeq\bm\la$.  
In particular, the partition $\bm\nu$ appearing in step 3(c) satisfies $\bm\mu\succ\bm\nu$; moreover, when $\alpha G^{\bm s}_e(\bm\nu)$ is subtracted from $A(\bm\mu)$, the condition that $a_{\bm\mu}=1$ and $a_{\bm\la}$ is non-zero only for $\bm\mu\succcurlyeq\bm\la$ remains true (because of Proposition \ref{decdom} and the fact that the order $\succcurlyeq$ refines the dominance order). So we can repeat step 3(c) ending up with the desired canonical basis vector $G^{\bm s}_e(\bm\mu)$. For examples, we refer the reader to \cite[Subsection 5.1]{Fay10} and \cite[Example 3.7]{DellA24b}.

\section{Addition of an empty runner}\label{sec:addEmptyRun}
{In this section, we build the combinatorial setting needed to prove our main result. Starting from the definition in \cite{Frunrem} and in \cite{DellA24b}, we define the addition of an \emph{empty} runner and prove some useful properties of adding empty runners. Set $e\geq 2$.}

\subsection{Addition of a runner for $r=1$}
In \cite{Frunrem}, Fayers defines the addition of a runner for the abacus display of a partition given a non-negative integer  $k$. We generalise this definition to any $k \in \Z$.

Given a partition $\la$, we construct a new partition $\la^{+k}$ for an integer $k$ as follows.
Let $a\geq l(\la)$ and construct the abacus configuration for $\la$ with charge $a$. Choose $k\in \Z$ such that $a+k\geq 0$, and write $a+k=ce+d$, with $c\ge0$ and $0 \leq d\leq e-1$. Add a runner to the abacus display $\mathrm{Ab}_e^a(\la)$ immediately to the left of runner $d$ with $c$ beads in the topmost $c$ positions, i.e. the position labelled $d, e+1+d, \ldots, (c-1)(e+1)+d$ in the usual labelling for an abacus with $e+1$ runners. The partition whose abacus display is obtained is $\la^{+k}$.

We now give a definition that partially motivates the title of the present paper. Beforehand we observe the following.

\begin{rmk}\label{smempty}
The smallest empty position in the abacus display of a partition $\la$ with charge $a$ is at position $a-l(\la)$.
\end{rmk}

\begin{defn}\label{emptyrunpar}
If the smallest empty position on the inserted runner of $\lak$ is the smallest space in the whole abacus display, then we say that the inserted runner is \emph{empty}.
\end{defn}

The following lemma gives an equivalent condition for an inserted runner to be empty. Consider a partition $\la$ and its abacus display with charge $a\ge l(\la)$. Take $k\in\Z$ such that $a+k\ge0$ and write $a+k=ce+d$ with $c\ge0$ and $0\le d\le e-1$. Construct the partition $\lak$ as above.

\begin{lem}\label{emptyeq}
    The inserted runner is empty if and only if $l(\la)\leq-k$.
\end{lem}

\begin{proof}
Suppose that the inserted runner is empty. Indeed, since the smallest empty position in the inserted runner is at $c(e+1)+d$, by \cref{smempty} we have that \[c(e+1)+d\le a+c-l(\lak).\] We observe that $l(\lak)\ge l(\la)$. Since the inserted runner is empty, all the beads occurring after the smallest space in the abacus of $\la$ still occur after the smallest space in the abacus of $\lak$. It follows that $\lak$ has at least the same number of parts of $\la$.

Hence we have that \[a+c-l(\la)\ge a+c-l(\lak)\ge c(e+1)+d=a+c+k,\]
from which we gain the desired inequality.

Conversely, by assumption we have that \[ce+d=a+k\le a-l(\la).\]
Together with \cref{smempty}, this implies that the smallest space in the abacus display of $\la$ is at least at the intersection of runner $d$ and level $c$. Since by definition the inserted runner of $\lak$ contains $c$ beads in the topmost positions and is inserted to the left of runner $d$, it follows that that the smallest space in the abacus of $\lak$ is exactly at position $c(e+1)+d$. Hence the inserted runner is empty. 
\end{proof}

We extend the operator $^{+k}$ linearly to the whole of the Fock space. We can now state an empty runner removal theorem for the Hecke algebra of the symmetric group $\mathfrak{S}_n$.

\begin{thm}\cite[Theorem 4.5]{JM02}\label{Emptyrunrem_level1} Let $\mu$ be a partition of $n$ and let $k \in \Z$ be such that $0 \leq l(\mu)\leq -k$. Then $G_{e+1}(\mu^{+k}) = G_e(\mu)^{+k}$.
\end{thm}

\subsection{Addition of an empty runner to the empty partition}
In this subsection we focus our attention on the addition of an empty runner to the empty partition $\varnothing$. We fix $k$ a negative integer and we consider the $e$-abacus display of $\varnothing$ with charge $a$ such that $a+k\ge0$. Thus, the Fock space we work with in this section is $\mathcal{F}^s$ where $s\in I$ and $a\equiv s\  (\mod  e)$. We start with recalling how this abacus display looks like.
\begin{rmk}\label{emptyab}
The $e$-abacus configuration of $\varnothing$ with charge $a$ has all the beads as high as possible within their runners (indeed, $\varnothing$ is an $e$-core): runners from $0$ to $i$ consist of $h+1$ beads and runners from $i+1$ to $e-1$ consist of $h$ beads, for some $h\geq 0$. Note that $i\equiv a-1\ (\mod e)$.
\end{rmk}

\begin{rmk}
The partition $\varnothing^{+k}$ is an $(e+1)$-core. This follows by construction because the inserted runner has all the beads as high as possible.
\end{rmk}

\begin{prop}\label{empty+k}
Let $-k=k_1e+k_2$ with $k_1\geq0$ and $0 \leq k_2 \leq e-1$.
Then $$
    \varnothing^{+k}=\begin{cases}
    ((k_1+1)^{k_2}, k_1^e, \ldots,2^e,1^e) & \text{ if }  k_2\neq 0;\\
    (k_1^e, \ldots,2^e,1^e) & \text{ if } k_2=0.
\end{cases}$$
\end{prop}


\begin{proof}
Let $i$ be the label of the runner of the largest bead in the abacus display $\mathrm{Ab}_e^a(\varnothing)$. We have $a-1=he + i$ for some $h\geq0$. Then $a=he+i+1$. By Remark \ref{emptyab}, the abacus display of $\varnothing$ looks like the following:
$$
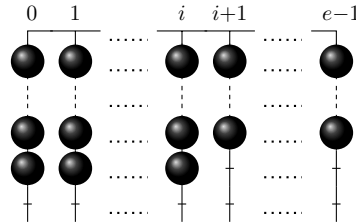
\begin{array}{ccccccc}
\mathclap{\scalebox{0.75}{0}}&\hspace{-0.4cm}\mathclap{\scalebox{0.75}{1}}&&\hspace{-0.25cm}\mathclap{\scalebox{0.75}{$i$}}&\hspace{-0.4cm}\mathclap{\scalebox{0.75}{$i{+}1$}}&&\mathclap{\scalebox{0.75}{$e{-}1$}}
\\
\sabacus(1.5,l,b,v,b,b,n)& \hspace{-0.4cm}
\sabacus(1.5,m,b,v,b,b,n)& \hspace{-0.25cm}
\sabacus(1.5,t,h,h,h,h,h)& \hspace{-0.25cm}
\sabacus(1.5,m,b,v,b,b,n)& \hspace{-0.4cm}
\sabacus(1.5,m,b,v,b,n,n)&\hspace{-0.25cm}
\sabacus(1.5,t,h,h,h,h,h)&\hspace{-0.25cm}
\sabacus(1.5,r,b,v,b,n,n)
\end{array}$$
with $h+1$ beads in runners from $0$ to $i$ and $h$ beads in runners from $i+1$ to $e-1$.



Suppose that $k_2=0$. If $i\neq e-1$, then $a+k=(h-k_1)e+i+1$. Hence, to get $\varnothing^{+k}$ we add a runner with $h-k_1$ beads in the topmost positions immediately to the left of runner $i+1$. Reading off the partition from $\mathrm{Ab}_{e+1}^{a+h-k_1}(\varnothing^{+k})$ we get the partition $(k_1^e, \ldots,2^e,1^e)$. Indeed, there are $k_1$ spaces before the largest bead in the abacus. Moreover, all these empty positions are in the inserted runner, so between two consecutive spaces there are exactly $e$ beads. If instead $i=e-1$, then $a+k=(h-k_1+1)e$ and a similar argument applies.

Now suppose that $k_2\ne0$. If $k_2\in \{1,\ldots,i+1\}$, then $a+k=(h-k_1)e+i+1-k_2$ and hence $\varnothing^{+k}$ is obtained by inserting a runner to the left of runner $i+1-k_2$ with $h-k_1$ beads in the topmost positions. Reading off the partition from $\mathrm{Ab}_{e+1}^{a+h-k_1}(\varnothing^{+k})$ we get the partition $((k_1+1)^{k_2}, k_1^e, \ldots,2^e,1^e)$. Indeed, there are $k_1+1$ spaces before the largest bead in the abacus, and all of them occur in the inserted runner. Moreover, there are $k_2$ beads after the largest of these spaces and $e$ beads between two of the other $k_1$ spaces. If $k_2\in \{i+2,\ldots,e-1\}$, then $a+k=(h-k_1-1)e+e+i+1-k_2$ and a similar argument applies.
\end{proof}

\begin{exe}\label{exempty+k}
Let $e=3$. Consider the abacus display of $\varnothing$ with charge $a=7$. 
\begin{enumerate}
\item Take $k=-3=-(1\cdot3+0)$. Then $\varnothing^{+(-3)}$ has the following abacus configuration
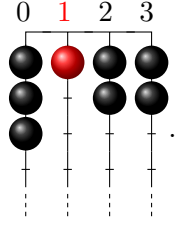
\begin{center}
\begin{tabular}{c}
        \hspace{-0.2cm}$\begin{matrix}
        0 &  \color{red}{1} &2 & 3
        \end{matrix}$\\
\sabacus(1.5,lmmr,bobb,bnbb,bnnn,nnnn,vvvv).
\end{tabular}
\end{center}
So, we have that $\varnothing^{+(-3)} = (1^3)$.

\item Take $k=-5=-(1\cdot3+2)$. Then $\varnothing^{+(-5)}$ has the following abacus configuration
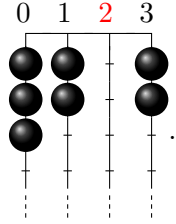
\begin{center}
\begin{tabular}{c}
        \hspace{-0.2cm}$\begin{matrix}
        0 & 1 & \color{red}{2} & 3
        \end{matrix}$\\
\sabacus(1.5,lmmr,bbnb,bbnb,bnnn,nnnn,vvvv).
\end{tabular}
\end{center}
So, we have that $\varnothing^{+(-5)} = (2^2,1^3)$.

\item Take $k=-7=-(2\cdot3+1)$. Then $\varnothing^{+(-7)}$ has the following abacus configuration
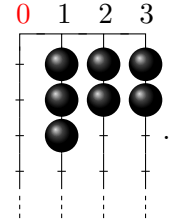
\begin{center}
\begin{tabular}{c}
        \hspace{-0.2cm}$\begin{matrix}
        \color{red}{0} & 1 & 2 & 3
        \end{matrix}$\\
\sabacus(1.5,lmmr,nbbb,nbbb,nbnn,nnnn,vvvv).
\end{tabular}
\end{center}
So, we have that $\varnothing^{+(-7)} = (3,2^3,1^3)$.

\end{enumerate}
\end{exe}

{\begin{rmk}\label{remnoderes}
All the removable nodes of $\varnothing^{+k}$ have the same residue because if there is any removable node in $\varnothing^{+k}$, then it corresponds to a bead to the right of the inserted runner and so it has residue $d+1\ (\mod e+1)$.
\end{rmk}


Given $j\in\Z$ we denote by $\mathfrak{F}_{j}$ the generator of the quantised enveloping algebra $U_v(\widehat{\mathfrak{sl}}_{e+1})$. 
Notice that when we write $\mathfrak{F}_j$ we mean $\mathfrak{F}_{\overline{j}}$ where $\overline{j}$ is the residue of $j$ modulo $e+1$.
Note that $\mathfrak{F}_{j}$ is bar-invariant for all $j\in\Z$. The next general fact is helpful in finding an induction sequence of operators from $\varnothing$ to $\varnothing^{+k}$.

\begin{lem}\label{Fneg_indstep}
    Let $\mu$ be a partition and $k<0$ such that $l(\mu)\leq -k$. Let $a\ge-k$ and consider the abacus display $\mathrm{Ab}_{e}^a(\mu)$. Write $a+k+1=ce+d$ with $c\ge0$ and $0\le d\le e-1$. Then
    $$\mu^{+k}=\mathfrak{F}_{d}^{(\ell)}(\mu^{+(k+1)}),$$
    for a constant $\ell$ depending on $k$ and the $e$-core of $\mu$.
\end{lem}

\begin{proof}

    If $d\neq 0$, the abacus display for $\mu^{+k}$ is obtained by inserting a runner containing $c$ beads in the topmost position to the immediate left of runner $d-1$, while the abacus display for $\mu^{+(k+1)}$ is obtained by adding such a runner to the immediate right of runner $d-1$ instead.
    If $d=0$, the abacus display for $\mu^{+k}$ is obtained by inserting a runner containing $c-1$ beads in the topmost positions to the immediate left of runner $e-1$, while the abacus display for $\mu^{+(k+1)}$ is obtained by inserting a runner with $c$ beads to the left of runner $0$, or equivalently, by inserting a runner with $c-1$ beads to the immediate right of runner $e-1$.
    Thus, in any case, the abacus display for $\mu^{+k}$ is obtained from the abacus display for $\mu^{+(k+1)}$ by swapping runner $d-1$ $(\mod e)$ and the inserted runner, i.e. simultaneously removing all removable $d$-nodes and adding all addable $d$-nodes of $\mu^{+(k+1)}$.
    
    By assumption $l(\mu)\le-k$ and so by \cref{emptyeq} the inserted runner of $\mu^{+k}$ is empty. It follows that the empty positions on runner $d-1$ of $\mathrm{Ab}_e^a(\mu)$ occur at levels $l\geq c$. Therefore $\mu^{+(k+1)}$ has no removable $d$-nodes and $\mu^{+k}$ is obtained by just adding all the addable $d$-nodes of $\mu^{+(k+1)}$. If $b$ denotes the number of beads on runner $d-1$ of $\mathrm{Ab}_e^a(\mu)$, then the number of addable $d$-nodes of $\mu^{+(k+1)}$ is $b-c$ (resp. $b-c+1$ if $d=0$).
    We conclude that $\mu^{+k}=\mathfrak{F}_{d}^{(\ell)}(\mu^{+(k+1)}),$ where $\ell=b-c$ (resp. $\ell=b-c+1$ if $d=0$). It is then clear that $\ell$ depends only on $k$, which gives $c$, and the core of $\mu$, which gives $b$.
\end{proof}

For $\ell\geq 1$ and $i\in \mathbb{Z}$, set
\begin{equation}
    \mathcal{G}^{(\ell)}_{i}:= \mathfrak{F}^{(\ell)}_{i-(e-1)} \ldots\mathfrak{F}^{(\ell)}_{i-1}  \mathfrak{F}^{(\ell)}_{i}.
\end{equation}
Observe that $\mathcal{G}^{(\ell)}_{i}$ is bar-invariant.

\begin{prop}\label{indseq}
Let $-k=k_1e+k_2$ with $k_1\geq0$ and $0 \leq k_2 \leq e-1$. Let $\alpha=a+k_1$. Then
    $$\mathfrak{F}^{(k_1+1)}_{{\alpha-k_2+1}} \ldots\mathfrak{F}^{(k_1+1)}_{{\alpha-1}}\mathfrak{F}^{(k_1+1)}_{{\alpha}}\mathcal{G}^{(k_1)}_{{\alpha-1}}\mathcal{G}^{(k_1-1)}_{{\alpha-2}} \ldots \mathcal{G}^{(1)}_{{a}}(\varnothing)=\varnothing^{+k},$$
where $\mathfrak{F}^{(k_1+1)}_{{\alpha-k_2+1}} \ldots\mathfrak{F}^{(k_1+1)}_{{\alpha-1}}\mathfrak{F}^{(k_1+1)}_{{\alpha}}$ occurs if and only if $k_2\neq0$.
\end{prop}

\begin{proof}
We prove this by induction on $-k\geq 0$.
\begin{itemize}
    \item If $k=0$, then $\varnothing^{+0}=\varnothing$ by \Cref{empty+k}, so there is nothing to prove.
    \item Suppose $-k>0$. 
    By \cref{Fneg_indstep}, $\varnothing^{+k} = \mathfrak{F}_{d}^{(\ell)}(\varnothing^{+(k+1)})$ with $d\equiv\alpha-k_2+1\ (\mod e+1)$ and $\ell\ge0$. In particular, by the proof of \Cref{empty+k} we know that
    $$\ell= \begin{cases}
    k_1 & \text{if } k_2=0, \\
    k_1+1 & \text{if } k_2\ne0.
    \end{cases}$$
    
    If $k_2=0$, then $-k-1=(k_1-1)e+e-1$ and so by the induction hypothesis we have 
         \begin{align*}    
         \varnothing^{+(k+1)}&=\mathfrak{F}_{{\alpha-e+1}}^{(k_1)}\cdots \mathfrak{F}_{{\alpha-2}}^{(k_1)} \mathfrak{F}_{{\alpha-1}}^{(k_1)} \mathcal{G}_{{\alpha-2}}^{(k_1-1)} \ldots \mathcal{G}_{{a}}^{(1)}(\varnothing).
         \end{align*}
         Hence, by \Cref{Fneg_indstep}
         \begin{align*}
         \varnothing^{+k}
         &= \mathfrak{F}_{{\alpha-e}}^{(k_1)} \mathfrak{F}_{{\alpha-e+1}}^{(k_1)}\cdots \mathfrak{F}_{{\alpha-2}}^{(k_1)} \mathfrak{F}_{{\alpha-1}}^{(k_1)} \mathcal{G}_{{\alpha-2}}^{(k_1-1)} \ldots \mathcal{G}_{{a}}^{(1)}(\varnothing)\\
         &= \mathcal{G}_{{\alpha-1}}^{(k_1)}\mathcal{G}_{{\alpha-2}}^{(k_1-1)}  \ldots \mathcal{G}_{{a}}^{(1)}(\varnothing).
         \end{align*}
         
    If $k_2\neq0$, then $-k-1=k_1e+k_2-1$ and so by the induction hypothesis we have 
        $$\varnothing^{+(k+1)}= \mathfrak{F}_{{\alpha-k_2+2}}^{(k_1+1)}\cdots \mathfrak{F}_{{\alpha-1}}^{(k_1+1)} \mathfrak{F}_{{\alpha}}^{(k_1+1)} \mathcal{G}_{{\alpha-1}}^{(k_1)} \mathcal{G}_{{\alpha-2}}^{(k_1-1)} \ldots \mathcal{G}_{{a}}^{(1)}(\varnothing).$$
         Again by \Cref{Fneg_indstep} it follows that
        \[\varnothing^{+k}=\mathfrak{F}_{{\alpha-k_2+1}}^{(k_1+1)}\mathfrak{F}_{{\alpha-k_2+2}}^{(k_1+1)}\cdots \mathfrak{F}_{{\alpha-1}}^{(k_1+1)} \mathfrak{F}_{{\alpha}}^{(k_1+1)} \mathcal{G}_{{\alpha-1}}^{(k_1)} \mathcal{G}_{{\alpha-2}}^{(k_1-1)} \ldots \mathcal{G}_{{a}}^{(1)}(\varnothing).\]
\end{itemize}
\end{proof}

\begin{exe}
With the same notation and choices of Example \ref{exempty+k}(iii), the sequence of operators in \cref{indseq} acts on $\varnothing$ as shown below.
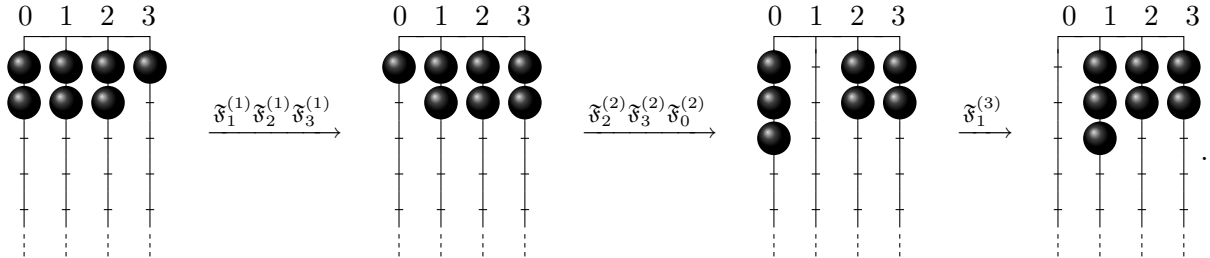
\begin{center}
\begin{tabular}{ccccccc}
  $\begin{array}{c}
  \begin{matrix}
        0 & 1 & 2 & 3
    \end{matrix}\\
  \sabacus(1.5, lmmr,bbbb,bbbn,nnnn,nnnn,nnnn,vvvv)
  \end{array}$
  &
  $\xrightarrow{\mathfrak{F}_{{1}}^{(1)}\mathfrak{F}_{2}^{(1)}\mathfrak{F}_{3}^{(1)}}$
  &
  $\begin{array}{c}
  \begin{matrix}
        0 & 1 & 2 & 3
    \end{matrix}\\ \sabacus(1.5, lmmr,bbbb,nbbb,nnnn,nnnn,nnnn,vvvv)
    \end{array}$
    &
    $\xrightarrow{\mathfrak{F}_{{2}}^{(2)}\mathfrak{F}_{{3}}^{(2)}\mathfrak{F}_{{0}}^{(2)}}$
    &
    $\begin{array}{c}
    \begin{matrix}
        0 & 1 & 2 & 3
    \end{matrix}\\ \sabacus(1.5, lmmr,bnbb,bnbb,bnnn,nnnn,nnnn,vvvv)
    \end{array}$
    &
    $\xrightarrow{\mathfrak{F}_{{1}}^{(3)}}$
    &
    $\begin{array}{c}
    \begin{matrix}
        0 & 1 & 2 & 3
    \end{matrix}\\ \sabacus(1.5, lmmr,nbbb,nbbb,nbnn,nnnn,nnnn,vvvv).
    \end{array}$
\end{tabular}
\end{center}
     
%
%
%
%
%
%
%
%
%
%
%
%
%
%
%
%
%
%
\end{exe}

The next lemma gives a way to establish when a certain sequence of operators acts as non-zero on the partition $\varnothing^{+(k+e+1)}$.

\begin{lem}\label{seqindnot0}
    Let $k\leq -e-1$. Write $-k=k_1e+k_2$ with $k_1\geq1$ and $0\leq k_2 \leq e-1$. Suppose that 
    \begin{equation}\label{action0}
    \mathfrak{F}_{{d-e}}^{(q_{e+1})}\mathfrak{F}_{{d-e+1}}^{(q_{e})} \ldots \mathfrak{F}_{{d-1}}^{(q_{2})}\mathfrak{F}_{{d}}^{(q_{1})} ( \varnothing^{+(k+e+1)} )\ne0,
    \end{equation}
    for some $q_j\geq 0$ for all $1 \leq j\leq e+1$, and $d$ as above. Then, $k_1\ge q_1\ge q_{e+1}-1$. In particular, $q_1=q_{e+1}-1$ if and only if $q_1=\ldots=q_{e+1-k_2}=k_1$ and $q_{e+2-k_2}=\ldots=q_{e+1}=k_1+1$.

    \begin{proof}
        We first note that the partition $\varnothing^{+(k+e+1)}$ has $k_1$ addable $ d$-nodes. Since by hypotesis the operator $\mathfrak{F}_{{d}}^{(q_{1})}$ acts as non-zero on $\varnothing^{+(k+e+1)}$, it follows that $q_1\le k_1$.

        Let $\hat{\la}\vcentcolon=\mathfrak{F}_{{d-e+1}}^{(q_{e})} \ldots \mathfrak{F}_{{d-1}}^{(q_{2})}\mathfrak{F}_{{d}}^{(q_{1})} ( \varnothing^{+(k+e+1)} )$. By hypotesis $\hat{\la}\ne0$ and $\mathfrak{F}_{{d-e}}^{(q_{e+1})} (\hat{\la})\ne0$. We observe that $\hat{\la}$ has at most $q_1+1$ addable $(d-e)$-nodes. Indeed, an addable $(d-e)$-node of $\hat{\la}$ corresponds either to a bead moved by the operator $\mathfrak{F}_{{d}}^{(q_{1})}$ at the beginning of the process, or to the bead at position $(c-1)(e+1)+d$. It follows that $q_{e+1}\le q_1+1$ proving the second inequality in the first statement.

        The second statement follows by looking at the abacus display of $\varnothing^{+(k+e+1)}$ and observing that $q_1=q_{e+1}-1$ if and only if each operator inductively adds the maximum number of addable nodes at each step. By the first paragraph of the proof, the equalities follow.
    \end{proof}
\end{lem}

\subsection{Addition of a multirunner for $r\geq2$}\label{def+k}

In \cite{DellA24b}, the first author defines the addition of a \textit{full} multirunner for a multipartition. Here, we generalise this definition to the addition of a generic multirunner. Consider the Ariki-Koike algebra $\aks$ and let $\bm\la$ be a $r$-multipartition of $n$. We construct a new $r$-multipartition as follows. 

Let $d\in I$. Let $\bm k=(k^{(1)},\dots,k^{(r)})\in\Z^r$ such that $k^{(j)}+s_j\equiv d\ (\mod e)$ for each $1\le j\le r$. Then let $\bm a=(a_1,\dots,a_r)\in\Z^r$ be a multicharge for $\aks$ such that $a_j\ge\mathrm{max}\{l(\la^{(j)}),k^{(j)}\}$ for each $1\le j\le r$, and consider the abacus display for $\bla$ with multicharge $\bm a$. For each $j$, write \[a_j+k^{(j)}=c_je+d\] with $c_j\ge0$ and $0\le d\le e-1$, and add a runner immediately to the left of runner $d$ of $\mathrm{Ab}_e^{a_j}(\la^{(j)})$ with $c_j$ beads in the topmost positions, i.e. the positions labelled by $d,e+1+d,\dots,(c_j-1)(e+1)+d$ in the usual labelling for an abacus display with $e+1$ runners. Note that all the new runners have the same label $d$, hence it makes sense to talk about the inserted multirunner. We denote by $\bm\la^{+\bm k}$ the $r$-multipartition whose abacus is obtained with the above procedure. Observe that $\bla^{+\bm k}=((\lambda^{(1)})^{+k^{(1)}},\dots,(\lambda^{(r)})^{+k^{(r)}})$ labels a dual Specht module for a different Ariki-Koike algebra whose $q$-connected parameters we denote by $\bm s^{+}=(s_1^+,\dots,s_r^+)\in(\Z/(e+1)\Z)^r$.

We extend the operator $^{+\bm k}$ linearly to the whole Fock space.

\begin{rmk}
Also for the addition of a generic multirunner, as for the one of a full multirunner (see \cite{DellA24b}), we can notice that $\bla^{+\bm k}$ does not depend on the choice of the multicharge.
\end{rmk}}

We now extend the definition of an empty runner to multipartitions.

\begin{defn}\label{def_empty_mlt}
    We say that the inserted multirunner of $\bla^{+\bm k}$ is \emph{empty} if for each component the inserted runner is empty according to \cref{emptyrunpar}.
\end{defn}

\begin{exe}\label{exrunnadd}
Suppose $\bm\la = ((4,3,2),(2^2),(3))$ and $e=4$. Let $\bm k=(-3,-9,-4)$ and choose $\bm a=(11,9,12)$. 
We get the following abacus display for $\bla$:
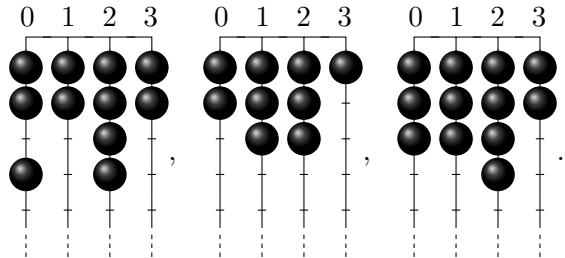
\begin{center}
\begin{tabular}{ccc}
        \hspace{-0.1cm}$\begin{matrix}
        0 &  1 & 2 & 3
        \end{matrix}$ & \hspace{-0.1cm}$\begin{matrix}
        0 &  1 & 2 & 3
        \end{matrix}$ & \hspace{-0.1cm}$\begin{matrix}
        0 &  1 & 2 & 3
        \end{matrix}$\\
\sabacus(1.5,lmmr,bbbb,bbbb,nnbn,bnbn,nnnn,vvvv), & 
\sabacus(1.5,lmmr,bbbb,bbbn,nbbn,nnnn,nnnn,vvvv), & 
\sabacus(1.5,lmmr,bbbb,bbbb,bbbn,nnbn,nnnn,vvvv).
\end{tabular}
\end{center}

Then $\bm\la^{+\bm k}=((6,5,3),(5^2,2^3,1^4),(5,1^3)$ and its abacus configuration is the following (with the inserted multirunner in red)
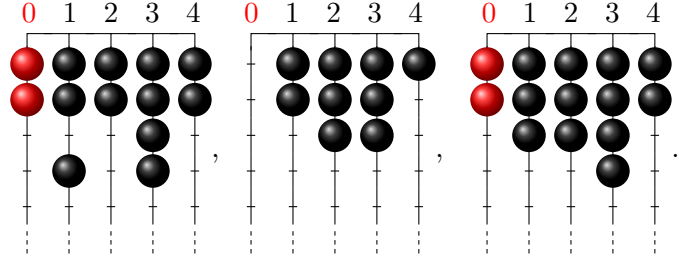
\begin{center}
\begin{tabular}{ccc}
        \hspace{-0.1cm}$\begin{matrix}
        \color{red}{0} &  1 & 2 & 3 & 4
        \end{matrix}$ & \hspace{-0.3cm}$\begin{matrix}
        \color{red}{0} &  1 & 2 & 3 & 4
        \end{matrix}$ & \hspace{-0.1cm}$\begin{matrix}
        \color{red}{0} &  1 & 2 & 3 & 4
        \end{matrix}$\\
\sabacus(1.5,lmmmr,obbbb,obbbb,nnnbn,nbnbn,nnnnn,vvvvv), & 
\sabacus(1.5,lmmmr,nbbbb,nbbbn,nnbbn,nnnnn,nnnnn,vvvvv), & 
\sabacus(1.5,lmmmr,obbbb,obbbb,nbbbn,nnnbn,nnnnn,vvvvv).
\end{tabular}
\end{center}

Notice that the inserted multirunner is empty.

\end{exe}

Now, we show some properties of the operator $^{+\bm k}$. 
In order to do this, we need the following definition: if $\bm{\lambda}$ is a multipartition and each component of $\bm\lambda$ is an $e$-core, then we say that $\bm{\lambda}$ is an $e$-\textit{multicore}. {Given $\bm\la \in \mathcal{P}^r$,
we define the $e$-\textit{multicore of} $\bm\la$ as the multipartition obtained by replacing each component
of $\bm\la$ with its $e$-core.}

\begin{lem}\label{multireg_pres}
Let $\bm k=(k^{(1)},\dots,k^{(r)})\in \Z^r$ and let $\bm\la$ be an $r$-multipartition. If $\bm\la^{+\bm k}$ is the multipartition constructed as above and the inserted multirunner is empty, then
\begin{enumerate}
    \item[$(i)$] $\bm\la^{+\bm k}$ is an $(e+1)$-multiregular $r$-multipartition.
    \item[$(ii)$] $\bm\la$ and $\bm\mu$ have the same $e$-multicore if and only if $\bm\la^{+\bm k}$ and $\bm\mu^{+\bm k}$ have the same $(e+1)$-multicore.
\end{enumerate}
\end{lem}

\begin{proof}
Recall that a multipartition is $(e+1)$-multiregular if and only if each of its component is $(e+1)$-regular, that is if and only if each of its component has at most $e$ equal parts (or equivalently, has at most $e$ consecutive bead positions in an abacus display).

Since we are assuming that $\bm\la^{+\bm k}$ is obtained from $\bm\la$ adding an empty multirunner, in every component of $\bm\la^{+\bm k}$ the smallest space occurs in the inserted runner. It follows that for each $1\le j \le r$ the partition ${(\la^{(j)})^{+k^{(j)}}}$ has at most $e$ consecutive beads in its abacus display and so it is a $(e+1)$-regular partition. This proves $(i)$.

For each $1\le j \le r$, recall that the $e$-abacus for the $e$-core of $\la^{(j)}$ is obtained {by sliding the beads up on their runner so that no bead has an empty position above it}. Hence, if $\gamma^{(j)}$ is the $e$-core of $\la^{(j)}$ then $(\gamma^{(j)})^{+k^{(j)}}$
 is the $(e + 1)$-core of $(\la^{(j)})^{+k^{(j)}}$, so $(ii)$ follows.
\end{proof}

\subsection{Induction operators and addition of an empty runner}
In this section, we give some results that will help to prove our main theorem (Theorem \ref{canbasis_runrem}).

We work in the following setting. Let $\bm\la$ be an $r$-multipartition. Let $\bm k,\bm a\in\Z^r$ and $\bla^{+\bm k}$ be as described at the beginning of Section \ref{def+k}. Recall that by construction the inserted multirunner is labeled by $d$ in an abacus display with $e+1$ runners. We define the following function that shows how to relabel the runners after the addition of the new multirunner. Let $g:\{0, \ldots, e-1\}\rightarrow\{0, \ldots, e\}\setminus \{d\}$ such that 
$$
g(i)= \begin{cases}
i & \text{if } i\in\{0, \ldots, d-1\}, \\
i+1 & \text{if } i\in\{d, \ldots, e-1\}.
\end{cases}$$

The following results describe the interaction between induction operators and the addition of a runner. Here, we follow the proof of \cite[Lemma 3.4]{JM02}

\begin{lem}\label{newrunnernotdd+1}
Let $\bm\la$ and $\bm\xi$ be $r$-multipartitions. Let $i \in I\setminus \{d\}$. Then $\bm\la \xrightarrow{m:i} \bm\xi$ if and only if $\bm\la^{+\bm k} \xrightarrow{m:g(i)} \bm\xi^{+\bm k}$, and if this happens, then  $N_i(\bm\la, \bm\xi) = N_{g(i)}(\bm\la^{+\bm k}, \bm\xi^{+\bm k})$.
\end{lem}
\begin{proof}
We have $\bm\la \xrightarrow{m:i} \bm\xi$ if and only if the abacus display for $\bm\xi$ is obtained from that of $\bla$ by moving $m$ beads from multirunner $i-1$ to multirunner $i$. In this case, $N_i(\bm\la, \bm\xi)$ is determined by the configurations of these multirunners in the abacus displays of $\bm\la$ and $\bm\xi$. The fact that $i\neq d$ means that in constructing the abacus displays for $\bm\la^{+\bm k}$ and $\bm\xi^{+\bm k}$ the inserted multirunner is not added in between multirunners $i-1$ and $i$. It follows that $\bm\la^{+\bm k} \xrightarrow{m:g(i)} \bm\xi^{+\bm k}$ and the coefficient $N_{g(i)}(\bm\la^{+\bm k}, \bm\xi^{+\bm k})$ is determined in exactly the same way as $N_i(\bm\la, \bm\xi)$.
\end{proof}

\begin{cor}\label{corineqd}
Let $m \geq 1$ and $i \in I \setminus \{d\}$. Then $\left(f_{i}^{(m)}(\bm\la)\right)^{+\bm k} =\mathfrak{F}_{{g(i)}}^{(m)}(\bm\la^{+\bm k})$.
\end{cor}
\begin{proof}
The thesis follows from Lemma \ref{newrunnernotdd+1} and the description of the action of $f_{i}^{(m)}$ given in Section \ref{algsec}.
\end{proof}

For the next lemma, we introduce the following notation. If $\bm\la$ and $\bm\xi$ are $r$-multipartitions, then we write $\bm\la \xrightarrow[m:i+1]{m:i} \bm\xi$ to indicate that $\bm\xi$ is obtained from $\bm\la$ by adding $m$ addable $i$-nodes and then $m$ addable $(i+1)$-nodes. This notation is just a shorter version of the following one:
$$\bm\la \xrightarrow{m:i}\bm\nu\xrightarrow{m:i+1} \bm\xi$$
where $\bm\nu$ is the $r$-multipartition obtained from $\la$ by adding $m$ addable $i$-nodes.
 
\begin{lem}\label{newrunnerdd+1}
Let $\bm\la$ and $\bm\xi$ be $r$-multipartitions and consider the multipartition $\bla^{+\bm k}$ constructed as above. If the inserted multirunner $d$ is empty, then $\bm\la \xrightarrow{m:d} \bm\xi$ if and only if $\bm\la^{+\bm k} \xrightarrow[m:d+1]{m:d} \bm\xi^{+\bm k}$.
\end{lem}

\begin{proof}
Suppose that $\bm\la \xrightarrow{m:d} \bm\xi$. This means that the abacus display for $\bm\xi$ is obtained by moving $m$ beads from multirunner $d-1$ to multirunner $d$ of $\mathrm{Ab}_e^{\bm a}(\bm\la)$. Notice that this moving of beads can occur simultaneously in different components of $\bm\la$, say that $\la^{(j_1)}, \ldots, \la^{(j_s)}$ are the components of $\bm\la$ involved to get $\bm\xi$. Let $m_t$ be the number of beads moved in $\la^{(j_t)}$, for $t=1, \ldots, s$.  So, for each $t=1, \ldots, s$, there are at least $m_t$ levels, say $\ell_1^{(t)}, \ldots, \ell_{m_t}^{(t)}$, in the abacus of $\la^{(j_t)}$ that present a configuration of the type
$$\begin{array}{cc}
 d-1 & d \\
 \abacus(b) &\abacus(n)
\end{array}.$$
When we apply the operator $^{+\bm k}$ to $\bm\la$, by assumption we add a empty runner in between runners $d-1$ and $d$ in each component of $\bm\la$. This implies that at these levels of the abacus of $\bla^{+\bm k}$ we have an abacus configuration of the type 
$$\begin{array}{ccc}
 d-1 & d & d+1 \\
 \abacus(b) &\abacus(n)& \abacus(n)
\end{array}.$$

Now consider $\bm\nu$ the unique $r$-multipartition obtained from the abacus of $\bla^{+\bm k}$ moving the $m$ beads in the components $j_1, \ldots,j_s$ and at levels $\ell_1^{(t)}, \ldots, \ell_{m_t}^{(t)}$ for $t=1, \ldots, s$. Then $\bm\nu \xrightarrow{m:d+1} \bm\xi^{+\bm k}$. Indeed, since the inserted multirunner is empty, the only beads that can be moved from multirunner $d$ to multirunner $d+1$ in the abacus of $\bm\nu$ are those moved at the first step.

Conversely, suppose that $\bm\la^{+\bm k} \xrightarrow[m:d+1]{m:d} \bm\xi^{+\bm k}$. Let $\bm\nu$ such that $\bm\la^{+\bm k} \xrightarrow{m:d}\bm\nu\xrightarrow{m:d+1} \bm\xi^{+\bm k}$. We want to show that $\bm\la \xrightarrow{m:d} \bm\xi$. Since the inserted multirunner is empty, the abacus configuration of an addable $d$-node of $\bm\la^{+\bm k}$ is one of the following:
\begin{enumerate}
\item[(1)] $$\begin{array}{ccc}
 d-1 & d & d+1 \\
 \abacus(b) &\abacus(n)& \abacus(b)
\end{array},$$
\item[(2)] $$\begin{array}{ccc}
 d-1 & d & d+1 \\
 \abacus(b) &\abacus(n)& \abacus(n)
\end{array}.$$
\end{enumerate}

If one of the beads moved from multirunner $d-1$ of the abacus of $\bla^{+\bm k}$ is of type (1), then in the abacus of $\bm\nu$ there are less than $m$ beads to move from multirunner $d$ to multirunner $d+1$. Hence (the abacus of) $\bm\xi^{+\bm k}$ cannot be obtained. It follows that all the beads we move from multirunner $d-1$ of the abacus of $\bm\la^{+\bm k}$ to get $\bm\nu$ come from an abacus configuration of type (2). Moreover, these $m$ beads are the only beads that can be moved from multirunner $d$ to multirunner $d+1$ afterwards. We conclude that if $j_1, \ldots, j_s$ are the components of $\bm\la^{+\bm k}$ involved to get $\bm\xi^{+\bm k}$ and the beads moved in the component $j_t$ are at levels $\ell_1^{(t)}, \ldots, \ell_{m_t}^{(t)}$, for $t=1, \ldots, s$, then $\bm\xi$ must be obtained from $\bm\la$ moving exactly the same beads from multirunner $d-1$ to multirunner $d$.
\end{proof}

\begin{lem}\label{Ndd+1}
With the same assumption of Lemma \ref{newrunnerdd+1}, we have $$N_d(\bm\la, \bm\xi) = N_{d}(\bm\la^{+\bm k}, \bm\nu) + N_{d+1}(\bm\nu,\bm\xi^{+\bm k})$$ where $\bm\nu$ is the unique $r$-multipartition such that $\bm\la^{+\bm k} \xrightarrow{m:d}\bm\nu\xrightarrow{m:d+1} \bm\xi^{+\bm k}$.
\end{lem}

\begin{proof}
We start noting that $\bla^{+\bm k}$ may have more addable $d$-nodes than $\bla$: these correspond to pairs of adjacent bead positions on multirunners $d-1$ and $d$ of the $e$-abacus of $\bm\la$ (configurations of type (1) in the proof of \cref{newrunnerdd+1}). 
As shown in the last paragraph of the proof of \cref{newrunnerdd+1}, none of these $d$-nodes is added to $\bla^{+\bm k}$ to get $\bm\nu$. It follows that each node $\mathfrak{n}\in\bm\xi\setminus\bla$ is in one to one correspondence with a node $\mathfrak{n}'\in\bm\nu\setminus\bla^{+\bm k}$ and a node $\mathfrak{n}''\in\bm\xi^{+\bm k}\setminus\bm\nu$.

Recall that by definition
\begin{align*}
N_d(\bm\la, \bm\xi) = \sum_{\mathfrak{n} \in \bm\xi \setminus \bm\la} \addab{d}{\bm\xi}-\remab{d}{\bla},
\end{align*}
\begin{align*}
N_{d}(\bm\la^{+\bm k}, \bm\nu) = \sum_{\mathfrak{n}' \in \bm\nu \setminus \bm\la^{+\bm k}} \addabp{d}{\bm\nu}-\remabp{d}{\bla^{+\bm k}},
\end{align*}
\begin{align*}
N_{d+1}(\bm\nu, \bm\xi^{+\bm k}) = \sum_{\mathfrak{n}'' \in \bm\xi^{+\bm k} \setminus \bm\nu} \addabs{d+1}{\bm\xi^{+\bm k}}-\remabs{d+1}{\bm\nu}.
\end{align*}

Now, consider $\mathfrak{n}\in \bm\xi\setminus \bm\la$ and let $J_{\mathfrak{n}}$ be the component of $\bm\xi$ of the node $\mathfrak{n}$. Set 
\begin{itemize}
\item $r_1$ to be the number of levels of the abacus of $\la^{(J_{\mathfrak{n}})}$ corresponding to nodes above $\mathfrak{n}$ with a configuration of type:
$$\begin{array}{cc}
 d-1 & d  \\
 \abacus(b) &\abacus(b)
\end{array},$$
\item $r_2$ to be the number of levels of the abacus of $\la^{(J_{\mathfrak{n}})}$ corresponding to nodes above $\mathfrak{n}$ with a configuration of type:
$$\begin{array}{cc}
 d-1 & d  \\
 \abacus(b) &\abacus(n)
\end{array},$$
\item $r_3$ to be the number of levels of the abacus of $\la^{(J_{\mathfrak{n}})}$ corresponding to nodes above $\mathfrak{n}$ with a configuration of type:
$$\begin{array}{cc}
 d-1 & d  \\
 \abacus(n) &\abacus(b)
\end{array},$$
\item $b$ to be the number of removable $d$-nodes in $\bm\xi\setminus\bm\la$ above $\mathfrak{n}$ in the component $J_{\mathfrak{n}}$.
\end{itemize}
In order to make easier to visualise the abacus display, we group together the levels of the abacus of $\la^{(J_{\mathfrak{n}})}$ above $\mathfrak{n}$ with the same configuration type, as shown below. This is illustrative only, the levels could easily appear in different order in general. Then the abacus display of $\la^{(J_{\mathfrak{n}})}$ and $\xi^{(J_{\mathfrak{n}})}$ look like the following (the bead that corresponds to the addition of $\mathfrak{n}$ is in \textcolor{red}{red}):
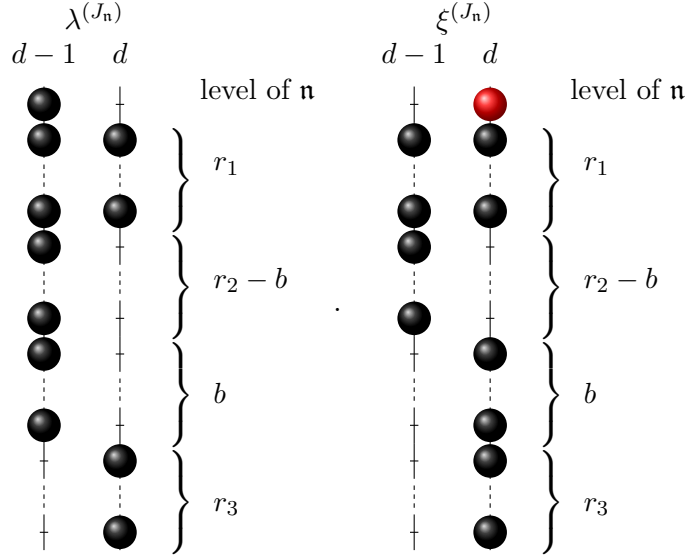
\begin{center}
\begin{tabular}{cc}
$\hspace{-2cm}\la^{(J_{\mathfrak{n}})}$ & $\hspace{-2cm}\xi^{(J_{\mathfrak{n}})}$\\
$\begin{array}{ccc}
 d-1 & d & \\
 \sabacus(1.5,b,b,v,b,b,v,b,b,v,b,n,v,n) &\sabacus(1.5,n,b,v,b,n,v,n,n,v,n,b,v,b) &
\begin{BMAT}(@)[0.5pt,0pt,4cm]{c}{ccccc}
     \vphantom{\rule{1mm}{14pt}}  \\
     \left.\vphantom{\rule{1mm}{21pt}} \right\rbrace \\ 
     \left.\vphantom{\rule{1mm}{21pt}} \right\rbrace \\ 
     \left.\vphantom{\rule{1mm}{21pt}} \right\rbrace \\ 
     \left.\vphantom{\rule{1mm}{21pt}} \right\rbrace 
\end{BMAT}
  \begin{BMAT}(@){l}{cc}
     \text{level of }\mathfrak{n} \\
     \begin{BMAT}(r)[5pt,2pt,6cm]{l}{cccc}
     r_1 \\ 
     r_2-b \\
     b\\
     r_3
     \end{BMAT}
  \end{BMAT} 
\end{array}.$
&

$\begin{array}{ccc}
 d-1 & d & \\
 \sabacus(1.5,n,b,v,b,b,v,b,n,v,n,n,v,n) &\sabacus(1.5,o,b,v,b,n,v,n,b,v,b,b,v,b) &
\begin{BMAT}(@)[0.5pt,0pt,6cm]{c}{ccccc}
     \vphantom{\rule{1mm}{14pt}}  \\
     \left.\vphantom{\rule{1mm}{21pt}} \right\rbrace \\ 
     \left.\vphantom{\rule{1mm}{21pt}} \right\rbrace \\ 
     \left.\vphantom{\rule{1mm}{21pt}} \right\rbrace \\ 
     \left.\vphantom{\rule{1mm}{21pt}} \right\rbrace 
\end{BMAT}
   \begin{BMAT}(@){l}{cc}
     \text{level of }\mathfrak{n} \\
     \begin{BMAT}(r)[5pt,2pt,6cm]{l}{cccc}
     r_1 \\ 
     r_2-b \\
     b\\
     r_3
     \end{BMAT}
  \end{BMAT} 
\end{array}.$
\end{tabular}
\end{center}
Thus, we have \[\addab{d}{\xi^{(J_{\mathfrak{n}})}}-\remab{d}{\la^{(J_{\mathfrak{n}})}}=(r_2-b)-r_3.\]
Now, consider the abacus displays of the component $J_{\mathfrak{n}}$ for the $r$-multipartitions $\bm\la^{+\bm k} \xrightarrow{m:d}\bm\nu\xrightarrow{m:d+1} \bm\xi^{+\bm k}$. Let $\mathfrak{n}'$ and $\mathfrak{n}''$ denote the nodes corresponding to $\mathfrak{n}$ (with the corresponding beads shown in \textcolor{red}{red}). Then, we have the following configurations:
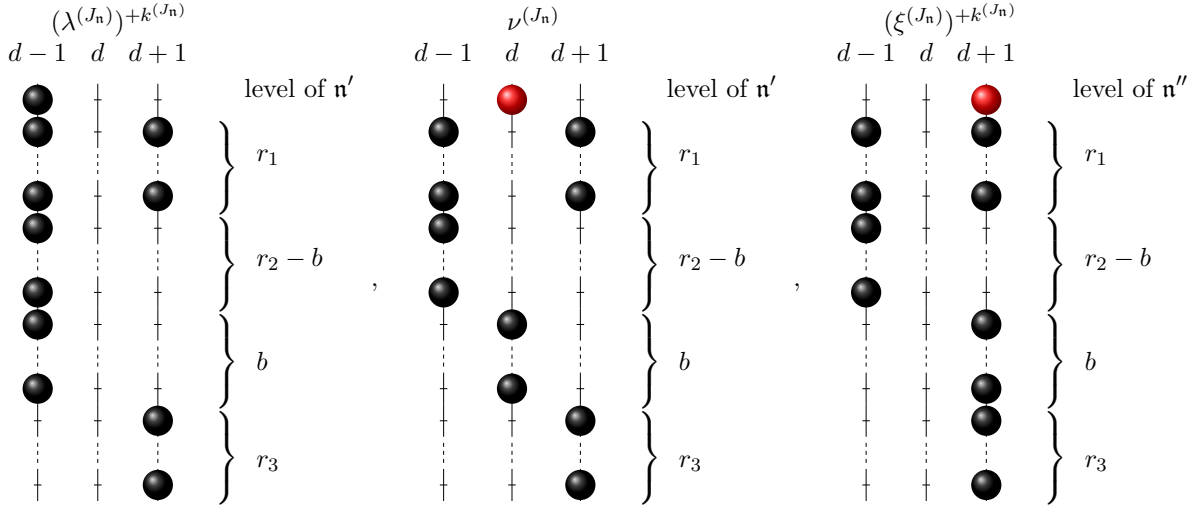
\begin{center}
\scalebox{0.9}{
\begin{tabular}{ccc}
$\hspace{-2cm}(\la^{(J_{\mathfrak{n}})})^{+k^{(J_{\mathfrak{n}})}}$ & $\hspace{-2cm}\nu^{(J_{\mathfrak{n}})}$ & $\hspace{-2cm}(\xi^{(J_{\mathfrak{n}})})^{+k^{(J_{\mathfrak{n}})}}$\\
$\begin{array}{cccc}
 d-1 & d & d+1 &\\
 \sabacus(1.5,b,b,v,b,b,v,b,b,v,b,n,v,n) & \sabacus(1.5,n,n,v,n,n,v,n,n,v,n,n,v,n) &\sabacus(1.5,n,b,v,b,n,v,n,n,v,n,b,v,b) & 
\begin{BMAT}(@)[0.5pt,0pt,6cm]{c}{ccccc}
     \vphantom{\rule{1mm}{14pt}}  \\
     \left.\vphantom{\rule{1mm}{21pt}} \right\rbrace \\ 
     \left.\vphantom{\rule{1mm}{21pt}} \right\rbrace \\ 
     \left.\vphantom{\rule{1mm}{21pt}} \right\rbrace \\ 
     \left.\vphantom{\rule{1mm}{21pt}} \right\rbrace 
\end{BMAT}
  \begin{BMAT}(@){l}{cc}
     \text{level of }\mathfrak{n}' \\
     \begin{BMAT}(r)[5pt,2pt,6cm]{l}{cccc}
     r_1 \\ 
     r_2-b \\
     b\\
     r_3
     \end{BMAT}
  \end{BMAT} 
\end{array},$
   &
   $\begin{array}{cccc}
 d-1 & d & d+1 &\\
 \sabacus(1.5,n,b,v,b,b,v,b,n,v,n,n,v,n) & \sabacus(1.5,o,n,v,n,n,v,n,b,v,b,n,v,n) &\sabacus(1.5,n,b,v,b,n,v,n,n,v,n,b,v,b) & 
\begin{BMAT}(@)[0.5pt,0pt,6cm]{c}{ccccc}
     \vphantom{\rule{1mm}{14pt}}  \\
     \left.\vphantom{\rule{1mm}{21pt}} \right\rbrace \\ 
     \left.\vphantom{\rule{1mm}{21pt}} \right\rbrace \\ 
     \left.\vphantom{\rule{1mm}{21pt}} \right\rbrace \\ 
     \left.\vphantom{\rule{1mm}{21pt}} \right\rbrace 
\end{BMAT}
   \begin{BMAT}(@){l}{cc}
     \text{level of }\mathfrak{n}' \\
     \begin{BMAT}(r)[5pt,2pt,6cm]{l}{cccc}
     r_1 \\ 
     r_2-b \\
     b\\
     r_3
     \end{BMAT}
  \end{BMAT} 
\end{array},$
   &
      $\begin{array}{cccc}
 d-1 & d & d+1 &\\
 \sabacus(1.5,n,b,v,b,b,v,b,n,v,n,n,v,n) & \sabacus(1.5,n,n,v,n,n,v,n,n,v,n,n,v,n) &\sabacus(1.5,o,b,v,b,n,v,n,b,v,b,b,v,b) & 
\begin{BMAT}(@)[0.5pt,0pt,6cm]{c}{ccccc}
     \vphantom{\rule{1mm}{14pt}}  \\
     \left.\vphantom{\rule{1mm}{21pt}} \right\rbrace \\ 
     \left.\vphantom{\rule{1mm}{21pt}} \right\rbrace \\ 
     \left.\vphantom{\rule{1mm}{21pt}} \right\rbrace \\ 
     \left.\vphantom{\rule{1mm}{21pt}} \right\rbrace 
\end{BMAT}
   \begin{BMAT}(@){l}{cc}
     \text{level of }\mathfrak{n}'' \\
     \begin{BMAT}(r)[5pt,2pt,6cm]{l}{cccc}
     r_1 \\ 
     r_2-b \\
     b\\
     r_3
     \end{BMAT}
  \end{BMAT} 
\end{array}.$  
\end{tabular}}
\end{center}
Thus, we have
\begin{gather*}
\addabp{d}{\nu^{(J_\mathfrak{n})}}-\remabp{d}{(\la^{(J_\mathfrak{n})})^{+k^{(J_\mathfrak{n}})}}=(r_2-b+r_1)-0,\\
\addabs{d+1}{(\xi^{(J_\mathfrak{n})})^{+k^{(J_\mathfrak{n})}}}-\remabs{d+1}{\nu^{(J_\mathfrak{n})}}=0-(r_1+r_3).
\end{gather*}
Taking the sum of the two  above equations, we obtain exactly $\addab{d}{\xi^{(J_{\mathfrak{n}})}}-\remab{d}{\la^{(J_{\mathfrak{n}})}}$.

By a similar argument, we can show that for every $j<J_\mathfrak{n}$,
\[\add{d}{\xi^{(j)}}-\rem{d}{\la^{(j)}}=\big[\add{d}{\nu^{(j)}}-\rem{d}{(\la^{(j)})^{+k^{(j)}}}\big]+\big[\add{d+1}{(\xi^{(j)})^{+k^{(j)}}}-\rem{d+1}{\nu^{(j)}}\big].\]
By Proposition \ref{Ncomps} we conclude.

\end{proof}

\begin{cor}\label{corieqd}
With the same assumption of Lemma \ref{newrunnerdd+1}, then \[\left(f_{d}^{(m)}(\bm\la)\right)^{+\bm k} =\mathfrak{F}_{{d+1}}^{(m)}\mathfrak{F}_{ d}^{(m)}(\bm\la^{+\bm k}).\]
\end{cor}
\begin{proof}
This is immediate from Lemmas \ref{newrunnerdd+1} and \ref{Ndd+1} and the description of the action of $f_{i}^{(m)}$ in Section \ref{algsec}.
\end{proof}


\section{Empty runner removal theorem}\label{sec:main thm}
In this section we prove an `empty runner removal' theorem for the Ariki-Koike algebra $\aks$. 

Fix $d\in I$. Let $\bm\mu=(\mu^{(1)},\dots,\mu^{(r)})$ be an $r$-multipartition of $n$. Let $\bm k=(k^{(1)},\dots,k^{(r)})\in(\Z_{\leq0})^{r}$ such that 
\begin{equation}\label{same_runn_d}
k^{(j)}+s_j\equiv d\ (\mod e)
\end{equation}
for all $1\le j\le r$, and
\begin{align}\label{conditiononk^j}
     k^{(1)}&\leq -l(\mu^{(1)}),  \\ 
     k^{(j)}&\leq  k^{(j-1)}-l(\mu^{(j)})-e, \text{      for all }1<j\leq r. \notag
\end{align}
Choose a multicharge $\bm a=(a_1,\dots,a_r)$ for $\aks$ such that $a_j\ge\mathrm{max}\{l(\mu^{(j)}),k^{(j)}\}$ for all $j$, and consider the abacus configuration $\mathrm{Ab}_e^{\bm a}(\bm\mu)$. Write \[a_j+k^{(j)}=c_je+d\] with $c_j\ge0$, and consider the multipartition $\bm\mu^{+\bm k}$ obtained as explained in Section \ref{def+k}. Note that the abacus display obtained has $\bm b=(b_1,\dots,b_r)$ beads where $b_j=a_j+c_j$ for all $1\le j\le r$. Note that $b_j\equiv s_j^+\ (\mod e+1)$ for all $j$.


\begin{rmk}\label{rmkempty}
    Conditions \eqref{conditiononk^j} on $\bm k$ imply that the inserted multirunner of $\bm\mu^{+\bm k}$ is empty by \cref{emptyeq} and \cref{def_empty_mlt}. 
\end{rmk}

We begin with a result that exhibits another useful implication of the conditions \eqref{conditiononk^j} on $\bm k$. Let $k\vcentcolon= k^{(1)}$ and write $-k=k_1e + k_2$ with $k_1\geq 0$ and $0 \leq k_2\leq e-1$. 
\begin{lem}\label{xle}
    Let $2\le j\le r$. The smallest empty position in the abacus display $\mathrm{Ab}_e^{a_j}(\mu^{(j)})$ is at least at level $c_j+k_1+1$.
\end{lem}

\begin{proof}
    Let $x$ be the level of the smallest empty position of $\mathrm{Ab}_e^{a_j}(\mu^{(j)})$. By \cref{smempty}, we have that $a_j-l(\mu^{(j)})=xe+i$ for some $x\ge0$ and $0\le i\le e-1$. Then,
    \begin{align*}
        x -c_j &= \dfrac{1}{e}(a_j-l(\mu^{(j)})-i)- \dfrac{1}{e}(a_j+k^{(j)}-d)\\
        & = \dfrac{1}{e}(-l(\mu^{(j)})-k^{(j)}-i+d)\\
        & \geq \dfrac{1}{e}(-k+\sum_{t=2}^{j-1} l(\mu^{(t)})+(j-1)e-i+d) \\
        & \geq \dfrac{1}{e}(-k+e-i+d) \\
        & =  \dfrac{1}{e}(k_1e+k_2+e-i+d)\\
        & =  k_1+ \dfrac{1}{e}(k_2+e-i+d)\\
        & >  k_1,
    \end{align*}
    and the thesis follows.
\end{proof}

Recall that $\bm\mu_-=(\mu^{(2)},\dots,\mu^{(r)})$ and $\bm\mu_0=(\varnothing,\bm\mu_-)$, and denote  $\bm k_-=(k^{(2)},\dots,k^{(r)})$. The next result shows that the sequence of operators of \cref{indseq} applied to the $r$-multipartition $(\varnothing,\bm\mu_-^{+\bm k_-})$ operates  only on the first component.

\begin{prop}\label{emptyto+k_r}
Let  $\alpha =b_1+k_1$. Then,
\begin{equation*}
\mathfrak{F}^{(k_1+1)}_{{\alpha-k_2+1}} \dots \mathfrak{F}^{(k_1+1)}_{{\alpha-1}}\mathfrak{F}^{(k_1+1)}_{{\alpha}}\mathcal{G}^{(k_1)}_{{\alpha-1}}\mathcal{G}^{(k_1-1)}_{{\alpha-2}} \ldots \mathcal{G}^{(1)}_{{b_1}}((\varnothing, \bm\mu_-^{+\bm k_-}))=\bm\mu_0^{+\bm k},    
\end{equation*}
where $\mathfrak{F}^{(k_1+1)}_{{\alpha-k_2+1}} \dots \mathfrak{F}^{(k_1+1)}_{{\alpha-1}}\mathfrak{F}^{(k_1+1)}_{{\alpha}}$ occurs if and only if $k_2\neq 0$.
\end{prop}

\begin{proof}
Consider the $(e+1)$-abacus configuration of the $r$-multipartition $(\varnothing,\bm\mu_-^{+\bm k_-})$ with multicharge $\bm b$.

Suppose $-e\le k\le -1$. By Proposition \ref{indseq} we have $$\varnothing^{+k}=
\mathfrak{F}^{(1)}_{{b_1+k+1}} \ldots \mathfrak{F}^{(1)}_{{b_1-1}}\mathfrak{F}^{(1)}_{{b_1}}
(\varnothing).
$$
Now we apply the same sequence of operators to $(\varnothing, \bm\mu_-^{+\bm k_-})$ and we show that the only resulting multipartition is $\bm\mu_0^{+\bm k}=(\varnothing^{+k},(\mu^{(2)})^{+k^{(2)}},\dots,(\mu^{(r)})^{+k^{(r)}})$. We first notice that ${b_1+k+1}\equiv d+1\ (\mod e+1)$, meaning that the residues of the operators can be rewritten from left to right as ${d+1},{d+2},\dots,{d-k}$. Note that none of these residues is congruent to $d\ (\mod e+1)$. If one operator, say $\mathfrak{F}^{(1)}_{{d+j}}$ for some $j\ge2$, acts as non-zero on $\bm\mu_-^{+\bm k_-}$, then the successive operator, $\mathfrak{F}^{(1)}_{{d+j-1}}$, must act on $\mathfrak{F}^{(1)}_{{d+j}}(\bm\mu_-^{+\bm k_-})$ because the partition $\mathfrak{F}^{(1)}_{{d+j+1}} \ldots\mathfrak{F}^{(1)}_{{d-k}}(\varnothing)$ has not addable $(d+j-1)$-nodes. Inductively, the last operator $\mathfrak{F}^{(1)}_{{d+1}}$ must act on $\mathfrak{F}^{(1)}_{{d+2}} \ldots\mathfrak{F}^{(1)}_{{d+j}}(\bm\mu_-^{+\bm k_-})$ and acts as 0. Indeed, the multirunner $d$ in the abacus of every term of $\mathfrak{F}^{(1)}_{{d+2}} \ldots\mathfrak{F}^{(1)}_{{d+j}}(\bm\mu_-^{+\bm k_-})$ is empty as it has not be affected by any of the operators, then none of the terms has addable $(d+1)$-nodes. We conclude that
$$\mathfrak{F}^{(1)}_{{d+1}} \ldots\mathfrak{F}^{(1)}_{{d-k}}(\varnothing,\bm\mu_-^{+\bm k_-}) =\bm\mu_0^{+\bm k}.$$

Suppose that $k<-e$ and hence that $k_1\ge1$. Working by induction on $-k$, we can assume that
\begin{align*}
\mathfrak{F}^{(k_1)}_{{\alpha-k_2+1}} \ldots\mathfrak{F}^{(k_1)}_{{\alpha-1}}\mathcal{G}^{(k_1-1)}_{{\alpha-2}} \ldots \mathcal{G}^{(1)}_{{b_1}}((\varnothing, \bm\mu_-^{+\bm k_-}))=(\varnothing^{+(k+e+1)}, \bm\mu_-^{+\bm k_-}).
\end{align*}
Thus, assuming $k_2\ne0$ (the case $k_2=0$ being similar), it is left to show that
\[\mathfrak{F}^{(k_1+1)}_{{\alpha-k_2+1}} \mathfrak{F}^{(k_1+1)}_{{\alpha-k_2+2}}\ldots\mathfrak{F}^{(k_1+1)}_{{\alpha}}\mathfrak{F}^{(k_1)}_{{\alpha+1}} \mathfrak{F}^{(k_1)}_{{\alpha+2}}\ldots\mathfrak{F}^{(k_1)}_{{\alpha-k_2}}(\varnothing^{+(k+e+1)}, \bm\mu_-^{+\bm k_-})=\bm\mu_0^{+\bm k}.\]

Suppose by contradiction that after applying the above induction sequence to $(\varnothing^{+(k+e+1)}, \bm\mu^{+\bm k})$, we have a non-zero term $(\xi,\bm\nu)\ne(\varnothing^{+k},\bm\mu_-^{+\bm k_-})$. Having that $\alpha-k_2\equiv d\ (\mod e+1)$, we relabel the residues of the $e+1$ operators as ${d+1},\dots,d$ from left to right.

For each $1\le i\le e+1$, let $p_i,q_i\ge0$ be such that
\[\mathfrak{F}_{{d+1}}^{(p_{e+1})}\ldots \mathfrak{F}_{{d}}^{(p_{1})}(\bm\mu_-^{+\bm k_-})=\bm\nu,\]
\[\mathfrak{F}_{{d+1}}^{(q_{e+1})}\ldots \mathfrak{F}_{{d}}^{(q_{1})}(\varnothing^{+(k+e+1)})=\xi.\]
Note that assuming that there is a non-zero term $(\xi, \bm\nu)\ne(\varnothing^{+k},\bm\mu_-^{+\bm k_-})$ is equivalent to saying that $p_i\ne0$ for some $i$. Moreover, observe that $$
p_i+q_i= \begin{cases}
k_1 & \text{if } i\in\{1,\ldots,e+1-k_2\}, \\
k_1+1 & \text{if } i\in\{e+2-k_2,\ldots,e+1\}.
\end{cases}$$
We claim that having a non-zero term $(\xi,\bm\nu)$ implies that $p_{e+1}\le p_1$. 


Let $\bm\nu_{1},\dots,\bm\nu_{e}$ be the $(r-1)$-multipartitions such that \[\bm\mu_-^{+\bm k_-}\xrightarrow{p_1:d} \bm\nu_{1}\xrightarrow{p_2:d-1}\dots\xrightarrow{p_e:d+2}\bm\nu_e\xrightarrow{p_{e+1}:d+1}\bm\nu.\]
The components of these multipartitions are indexed from 2 to $r$. For $2\le j\le r$, let $n_j$ be the number of beads moved to get $\nu_1^{(j)}$ from $(\mu^{(j)})^{+k^{(j)}}$ and $m_j$ the number of beads moved to get $\nu^{(j)}$ from $\nu_e^{(j)}$. Clearly, $n_2+\dots+n_r=p_1$ and $m_2+\dots+m_r=p_{e+1}$. We show that $m_j\le n_j$ for all $j$.

By \cref{rmkempty}, the (inserted) runner $d$ of $\mathrm{Ab}_{e+1}^{b_j}((\mu^{(j)})^{+k^{(j)}})$ is empty. Then, the movable $m_j$ beads on runner $d$ of $\mathrm{Ab}_{e+1}^{b_j}(\nu_e^{(j)})$ are either the $n_j$ beads moved at the first step or the bead at position $d+(c_j-1)(e+1)$. We deduce that $m_j\le n_j+1$. Suppose by contradiction that there is a component $J\in\{2,\dots,r\}$ such that $m_J=n_J+1$. We look at $\mathrm{Ab}_{e+1}^{b_j}((\mu^{(J)})^{+k^{(J)}})$. By \cref{xle}, after the bead at position $(c_J-1)(e+1)+d$ - the largest along the inserted runner - there are at least $k_1+1$ sequences of $e$ consecutive beads detached by spaces on runner $d$. Since we are assuming that the bead at $(c_J-1)(e+1)+d$ in $\mathrm{Ab}_{e+1}^{b_J}(\nu_e^{(J)})$ is followed by a space, we have that the bead at level $c_J$ of runner $d-1$ of $\mathrm{Ab}_{e+1}^{b_J}((\mu^{(J)})^{+k^{(J)}})$ must be one of the $n_J$ beads moved at the first step. Applying this reasoning inductively we find that $n_J\ge k_1+1$ having a contradiction because $n_J\le p_1\le k_1$. Thus, for every $2\le j\le r$, $m_j\le n_j$, and hence, $p_{e+1}\le p_1$.

Now we can conclude. Indeed, it follows that $q_{e+1}\ge q_1+1$. Lemma \ref{seqindnot0} then forces $q_{e+1}=q_1+1$ and says that $q_1=\dots=q_{e+1-k_2}=k_1$ and $q_{e+2-k_2}=\dots=q_{e+1}=k_1+1$ in this case. Therefore, $p_i=0$ for every $1\le i\le e+1$, giving the desired contradiction.
\end{proof}

The above proposition gives the following result on the canonical basis coefficients of $(\varnothing,\bm\mu_-^{+\bm k_-})$ and $\bm\mu_0^{+\bm k}=(\varnothing^{+ k},\bm\mu_-^{+\bm k_-})$.

\begin{prop}\label{canbasis_empty}
Suppose that
$$G^{\bm s^+}_{e+1}((\varnothing,\bm\mu_-^{+\bm k_-}))=\sum_{\bm\mu_-\trianglerighteq\bm\la_-}d_{\bm\la_-\bm\mu_-}^{\bm s_-}(v) (\varnothing,\bm\la_-^{+\bm k_-})$$
where $d_{\bm\la_-\bm\mu_-}^{\bm s_-}(v) \in v\mathbb{N}[v]$ for $\bm\la_- \neq \bm\mu_-$. Then,
$$G^{\bm s^+}_{e+1}(\bm\mu_0^{+\bm k})=\sum_{\bm\mu_-\trianglerighteq\bm\la_-}d_{\bm\la_-\bm\mu_-}^{\bm s_-}(v) \bm\la_0^{+\bm k}.$$
\end{prop}

\begin{proof}
We apply the sequence of operators of Proposition \ref{indseq} to $G^{\bm s^+}_{e+1}((\varnothing,\bm\mu_-^{+\bm k_-}))$:
\begin{equation}\label{actionG}
\mathfrak{F}^{(k_1+1)}_{{\alpha-k_2+1}}\ldots \mathfrak{F}^{(k_1+1)}_{{\alpha-1}}\mathfrak{F}^{(k_1+1)}_{{\alpha}}\mathcal{G}^{(k_1)}_{{\alpha-1}}\mathcal{G}^{(k_1-1)}_{{\alpha-2}} \ldots \mathcal{G}^{(1)}_{{b_1}}
(G^{\bm s^+}_{e+1}(\varnothing, \bm\mu_-^{+\bm k_-})),
\end{equation}
where $\alpha = b_1+k_1$ and $\mathfrak{F}^{(k_1+1)}_{{\alpha-k_2+1}}\ldots \mathfrak{F}^{(k_1+1)}_{{\alpha}}$ occurs if and only if $k_2\neq 0$.

By Proposition \ref{emptyto+k_r}, we have
\begin{align*}
    \eqref{actionG} &=\sum_{\bm\mu_-\trianglerighteq\bm\la_-}d_{\bm\la_-\bm\mu_-}^{\bm s_-}(v) \mathfrak{F}^{(k_1+1)}_{{\alpha-k_2+1}}\ldots \mathfrak{F}^{(k_1+1)}_{{\alpha-1}}\mathfrak{F}^{(k_1+1)}_{{\alpha}}\mathcal{G}^{(k_1)}_{{\alpha-1}}\mathcal{G}^{(k_1-1)}_{{\alpha-2}} \ldots \mathcal{G}^{(1)}_{{b_1}}((\varnothing,\bm\la_-^{+\bm k_-}))\\
    &=\sum_{\bm\mu_-\trianglerighteq\bm\la_-}d_{\bm\la_-\bm\mu_-}^{\bm s_-}(v) \bm\la_0^{+\bm k},
\end{align*}
which is of the form $\bm\mu_0^{+\bm k}+\sum\limits_{\bm\mu_-\triangleright\bm\la_-}d_{\bm\la_-\bm\mu_-}^{\bm s_-}(v)\bm\la_0^{+\bm k}$.
Hence, the desired equality follows by the uniqueness of the canonical basis vectors.
\end{proof}

We set some notation before stating the very last preliminary result. If $\mathfrak{f}={f}_{i_l}^{(h_l)} \cdots {f}_{i_1}^{(h_1)}$ is an operator with $h_1, \ldots, h_l\ge0$ and $i_1, \ldots, i_l \in I$, we denote by $\mathfrak{F}$ the operator obtained in the following way: for all $j=1, \ldots, l$
\begin{itemize}
    \item if $i_j\neq d$, replace ${f}_{i_j}^{(h_j)}$ with $\mathfrak{F}^{(h_j)}_{g(i_j)}$,
    \item if $i_j=d$, replace ${f}_d^{(h_j)}$ with $\mathfrak{F}^{(h_j)}_{d+1}\mathfrak{F}^{(h_j)}_{d}$.
\end{itemize}

\begin{prop}\label{same_coeff}

Let $\bm\mu_-$ be an $e$-multiregular $(r-1)$-multipartition. 
Let $\mathfrak{f}=f_{i_l}^{(h_l)} \cdots f_{i_1}^{(h_1)}$ be such that $\mathfrak{f}\cdot\varnothing = \mu^{(1)} +\sum\limits_{\mu^{(1)} \triangleright  \tau} t_{\tau}\tau$ for $t_{\tau}\in \Z[q, q^{-1}]$. Suppose that 
$$\mathfrak{f} \cdot G_e^{\bm s}(\bm\mu_0)=\sum_{\bm\nu\in \mathcal{P}^{r}}g_{\bm\nu}\bm\nu,$$
where $g_{\bm\nu}\in \mathbb{Z}[q, q^{-1}]$. Then, $$\mathfrak{F} \cdot G_{e+1}^{\bm s^{+}}(\bm\mu_0^{+\bm k })=\sum_{\bm\nu\in \mathcal{P}^{r}}g_{\bm\nu}\bm\nu^{+\bm k }.$$
\end{prop}

\begin{proof}
By the linearity of $^{+\bm k}$ we have that
$$\left(\mathfrak{f} \cdot G_e^{\bm s}(\bm\mu_0)\right)^{+\bm k}=\sum_{\bm\nu\in \mathcal{P}^{r}}g_{\bm\nu}\bm\nu^{+\bm k}.$$
Moreover, we get that
\begin{align*}
    \left(\mathfrak{f} \cdot G_e^{\bm s}(\bm\mu_0)\right)^{+\bm k}&=\left(\mathfrak{f} \cdot \sum_{\bm\mu_-\trianglerighteq\bm\la_-}d_{\bm\la_-\bm\mu_-}^{\bm s_-}(v) \bm\la_0 \right)^{+\bm k}&& \text{by Prop. \ref{fcomp}}\\
    &=\mathfrak{F} \cdot \sum_{\bm\mu_-\trianglerighteq\bm\la_-}d_{\bm\la_-\bm\mu_-}^{\bm s_-}(v) \bm\la_0^{+\bm k} && \text{by Cor. \ref{corineqd}, \ref{corieqd}}\\
    &=\mathfrak{F} \cdot G^{\bm s^+}_{e+1}(\bm\mu_0^{+\bm k }). && \text{by Prop. \ref{canbasis_empty}}
\end{align*}
Comparing the two equalities above, the thesis follows.
\end{proof}

We now come to the end of this paper. The next result shows that the operator $^{+\bm k}$ commutes with the canonical basis vectors. This directly implies \cref{main}.

\begin{thm}\label{canbasis_runrem}
Suppose that $\bm\mu$ is an $e$-multiregular multipartition.
Then $$G^{\bm s^{+}}_{e+1}(\bm\mu^{+\bm k})=G^{\bm s}_{e}(\bm\mu)^{+\bm k}.$$
\end{thm}

\begin{proof}
We proceed by induction on the number $r$ of components with base case being Theorem \ref{Emptyrunrem_level1}.

Suppose $r>1$. By the inductive hypothesis, for the $e$-multiregular $(r-1)$-multipartition $\bm\mu_-$ and $\bm s_-^+=(s_2^+, \ldots, s_r^+)$ we have
\begin{align*}
    G^{\bm s_-^{+}}_{e+1}(\bm\mu_-^{+\bm k_-})& = G^{\bm s_-}_{e}(\bm\mu_-)^{+\bm k_-}\\
    & =\bm\mu_-^{+\bm k_-} + \sum_{\bm\mu_-\triangleright\bm\la_-}d_{\bm\la_-\bm\mu_-}^{\bm s_-}(v) \bm\la_-^{+\bm k_-}.
\end{align*}

Now, we want to show that the thesis is also true for the $r$-multipartition $(\mu^{(1)},\bm\mu)$.
Then, by \cref{fcomp},
$$G^{\bm s^+}_{e+1}((\varnothing,\bm\mu_-^{+\bm k_-}))=(\varnothing,\bm \mu_-^{+\bm k_-}) + \sum_{\bm\mu_-\triangleright\bm\la_-}d_{\bm\la_-\bm\mu_-}^{\bm s_-}(v) (\varnothing,\bm\la_-^{+\bm k_-}).$$
Hence, by Proposition \ref{canbasis_empty} we have that
$$G^{\bm s^+}_{e+1}(\bm\mu_0^{+\bm k})=\bm\mu_0^{+\bm k}+ \sum_{\bm\mu_-\triangleright\bm\la_-}d_{\bm\la_-\bm\mu_-}^{\bm s_-}(v) \bm\la_0^{+\bm k}.$$

Using the LLT algorithm on partitions, we can write $G^{(s_1)}(\mu^{(1)})$ as $\mathfrak{f}\cdot\varnothing$ in the Fock space $\mathcal{F}^{(s_1)}$, for some $\mathfrak{f}\in \mathcal{U}$.
Applying the operator $\mathfrak{f}$ to $G_e^{\bm s}(\bm\mu_0)$ we can write
\begin{equation}\label{Ge(0,mu)}
   \mathfrak{f} \cdot G_e^{\bm s}(\bm\mu_0)=\sum_{\bm\nu\in \mathcal{P}^{r}}g_{\bm\nu}\bm\nu, 
\end{equation}
with $g_{\bm\nu}\in \mathbb{Z}[q, q^{-1}]$, because $\mathfrak{f} \cdot G_e^{\bm s}(\bm\mu_0)$ is an element of the $\mathcal{U}$-submodule $M^{\otimes \bm s}$ of $\mathcal{F}^{\bm s}$.
Performing step (c) of the LLT algorithm for multipartitions in \cite{Fay10} we get 
\begin{equation}\label{fGe(0,mu)}
  \mathfrak{f} \cdot G_e^{\bm s}(\bm\mu_0) - \sum_{\bm\mu \triangleright \bm\sigma} a_{\bm\sigma \bm\mu}(v) G_e^{\bm s}(\bm\sigma) = G_e^{\bm s}( \bm\mu) 
\end{equation}
for some $a_{\bm\sigma \bm\mu}(v) \in \mathbb{Z}[q+q^{-1}]$.

Now consider $\mathfrak{F}$ the operator obtained from $\mathfrak{f}$ by applying the procedure explained just before \cref{same_coeff}. By \cref{same_coeff} we have 
\begin{equation}\label{Ge+(0+,mu+)}
   \mathfrak{F} \cdot G_{e+1}^{\bm s^{+}}(\bm\mu_0^{+\bm k})=\sum_{\bm\nu\in \mathcal{P}^{r}}g_{\bm\nu}\bm\nu^{+\bm k}. 
\end{equation}
We perform the following subtraction of terms

\begin{equation}\label{alice}
    \mathfrak{F} \cdot G_{e+1}^{\bm s^{+}}(\bm\mu_0^{+\bm k}) - \sum_{\bm\mu \triangleright \bm\sigma} a_{\bm\sigma \bm\mu}(v) G_{e+1}^{\bm s^{+}}(\bm\sigma^{+\bm k}),
\end{equation}
and we proceed by induction on the dominance order supposing that $G_{e+1}^{\bm s^{+}}(\bm\sigma^{+\bm k}) = G_e^{\bm s}(\bm\sigma)^{+\bm k}$ for all $\bm\sigma \triangleleft \bm\mu$. Then
\begin{align}\label{indFG-G+}
\eqref{alice} 
= \mathfrak{F} \cdot G_{e+1}^{\bm s^{+}}(\bm\mu_0^{+\bm k}) - \sum_{\bm\mu\triangleright \bm\sigma} a_{\bm\sigma \bm\mu}(v) G_{e}^{\bm s}(\bm\sigma)^{+\bm k}.
\end{align}
Since in \eqref{indFG-G+} we are performing exactly the same operations as in \eqref{fGe(0,mu)} and all the involved coefficients are the same,
by linearity of $^{+ \bm k}$ we have 
\begin{equation*}
 \eqref{indFG-G+} = G_e^{\bm s}(\bm\mu)^{+\bm k}.   
\end{equation*}
Moreover, by the uniqueness of the canonical basis vectors of $M^{\otimes\bm s^+}$ we have that
$$ \eqref{indFG-G+}= G_{e+1}^{\bm s^+}(\bm\mu^{+\bm k}).$$
Hence, we conclude that
$$G^{\bm s^{+}}_{e+1}(\bm\mu^{+\bm k})=G^{\bm s}_{e}(\bm\mu)^{+\bm k}.$$
\end{proof}


\addcontentsline{toc}{chapter}{Bibliography}
\bibliographystyle{alpha}
\bibliography{bibfile}

\end{document}